\documentclass{amsart}

\usepackage{datetime}
\usepackage{amsfonts}
\usepackage{amsmath}
\usepackage{amssymb}
\usepackage{amsthm}
\usepackage{eucal}
\usepackage{graphicx,color}
\usepackage[usenames,dvipsnames,svgnames,table]{xcolor}

\usepackage{hyperref}
\hypersetup{colorlinks,
            filecolor=black,
            linkcolor=MidnightBlue,
            citecolor=NavyBlue,
            urlcolor=RoyalBlue,
            bookmarksopen=true}

\theoremstyle{plain}
\newtheorem{lemma}{Lemma}
\newtheorem{corollary}[lemma]{Corollary}

\newtheorem*{main}{Main Theorem}
\newtheorem{assumption}{Assumption}

\theoremstyle{remark}
\newtheorem*{acknowledgment}{Acknowledgment}
\newtheorem{remark}{Remark}

\newcommand{\cv}{\nabla}
\newcommand{\dn}{\check}
\newcommand{\Ds}{\partial_s}
\newcommand{\mb}{\mathbb}
\newcommand{\mc}{\mathcal}
\newcommand{\mr}{\mathrm}
\newcommand{\dx}{\mr{d}\xi}
\newcommand{\ds}{\mr{d}s}
\newcommand{\lh}{\big(}
\newcommand{\lp}{\langle}

\newcommand{\rh}{\big)_{\mc G}}
\newcommand{\rn}{\|_{\mc G}}
\newcommand{\rp}{\rangle}
\newcommand{\ten}{\otimes}
\newcommand{\up}{\hat}
\newcommand{\ve}{\varepsilon}
\newcommand{\vp}{\varphi}

\DeclareMathOperator{\Rm}{Rm}
\DeclareMathOperator{\Rc}{Rc}

\newcounter{mnotecount}[section]
\let\oldmarginpar\marginpar
\renewcommand\marginpar[1]{\-\oldmarginpar[\raggedleft\footnotesize #1]
{\raggedright\footnotesize #1}}

\begin{document}

\title[Ricci flow neckpinches without rotational symmetry]
{Ricci flow neckpinches without \\ rotational symmetry}

\author{James Isenberg}
\address[James Isenberg]{University of Oregon}
\email{isenberg@uoregon.edu}
\urladdr{http://www.uoregon.edu/$\sim$isenberg/}

\author{Dan Knopf}
\address[Dan Knopf]{University of Texas at Austin}
\email{danknopf@math.utexas.edu}
\urladdr{http://www.ma.utexas.edu/users/danknopf}

\author{Nata\v sa \v Se\v sum}
\address[Nata\v sa \v Se\v sum]{Rutgers University}
\email{natasas@math.rutgers.edu}
\urladdr{http://www.math.rutgers.edu/$\sim$natasas/}

\thanks{
JI thanks the NSF for support in PHY-1306441.
DK thanks the NSF for support in DMS-1205270.
N\v S thanks the NSF for support in DMS-0905749 and DMS-1056387.
JI also thanks the Mathematical Sciences Research Institute in Berkeley, California
for support under grant 0932078 000. Some of this work was carried out while JI was
in residence at MSRI during the fall of 2013.
}

\begin{abstract}
We study ``warped Berger'' solutions $\big(\mc S^1\times\mc S^3,G(t)\big)$
of Ricci flow: generalized warped products with the metric induced on
each fiber $\{s\}\times\mathrm{SU}(2)$ a left-invariant Berger metric.
We prove that this structure is preserved by the flow, that these solutions develop
finite-time neckpinch singularities, and that they asymptotically approach
round product metrics in space-time neighborhoods of their singular sets,
in precise senses. These are the first examples of Ricci flow solutions
without rotational symmetry that become asymptotically rotationally
symmetric locally as they develop local finite-time singularities.
\end{abstract}

\maketitle
\setcounter{tocdepth}{1}
\tableofcontents

\section{Introduction}
\label{Intro}
There are many examples of solutions of parabolic geometric \textsc{pde} that become round as they
develop global singularities: for instance, this phenomenon has been observed, in chronological
order, for $3$-manifolds of positive Ricci curvature evolving by Ricci flow \cite{Ha82}, for convex
hypersurfaces evolving by mean curvature flow \cite{Hu84}, for compact embedded solutions
of curve-shortening  flow \cite{GH86, Gr87}, and for $1/4$-pinched solutions of Ricci flow \cite{BS09}.
Although these examples are usually viewed as special cases of the propensity of geometric flows to 
asymptotically approach constant-curvature geometries, it can be informative instead to interpret them in
the spirit of Klein's \emph{Erlangen Program} as examples of symmetry enhancement along geometric flows,
with these solutions asymptotically acquiring  larger symmetry groups than the symmetry groups of their
initial data.

There is growing evidence that the same phenomenon holds locally in space-time neighborhoods
of local singularities. For example, rotationally-symmetric solutions of Ricci flow that develop
neckpinch singularities asymptotically acquire the additional translational symmetry of the
cylinder soliton \cite{AK04, AK07}. More recently, it has been shown that any complete
noncompact $2$-dimensional solution of mean curvature flow that is sufficiently $C^3$-close to a standard
round neck at some time will develop a finite-time singularity and become asymptotically rotationally
symmetric  in a space-time neighborhood of that singularity \cite{GKS12, GK13}. As well, numerical experiments
support the expectation that broader classes of mean curvature flow solutions asymptotically develop
additional local symmetries as they become singular \cite{GIK}. All of these results contribute to the
developing heuristic principle that singularities of parabolic geometric evolution equations are nicer than
one might naively expect.

In this paper, we obtain an analogous result for $4$-dimensional solutions of Ricci flow, but with comparatively
weaker hypotheses on the initial data than those used in \cite{GKS12, GK13}. We replace those hypotheses
with a structural assumption that the metrics under consideration are certain Riemannian submersions.
Specifically, we consider generalized warped product solutions
$\big(\mc S^1\times\mc S^3,G(t)\big)$, where
\begin{equation}	\label{SimpleAnsatz}
 G = (\ds)^2+\Big\{f^2\,\omega^1\ten\omega^1
 	   + g^2\big(\omega^2\ten\omega^2
	   + \omega^3\ten\omega^3\big)\Big\}.
\end{equation}
For each $s\in\mc S^1$, the quantity in braces is a left-invariant metric on the fiber
$\mathrm{SU}(2)\approx\mc S^3$ over $s$, written with respect to a coframe
$(\omega^1,\omega^2,\omega^3)$ that is algebraically dual to a fixed Milnor frame. We
provide a detailed description of these geometries, which we call \emph{warped Berger metrics,} 
in Appendix~\ref{Ansatz}. As shown there,  this \emph{Ansatz} is preserved by Ricci flow. We
prove that if a metric of this form develops a local ``neckpinch'' singularity, then the fibers become
asymptotically round in space-time neighborhoods of its singular sets. More precisely, we prove
progressively stronger results under these progressively stronger assumptions:

\begin{assumption}	\label{MildAssumption}
$\big(\mc S^1\times\mc S^3,G(t)\big)$  is a warped Berger solution of Ricci flow such that
\begin{enumerate} 
\item $f\leq g$ at $t=0$;
\item $\big\{\min_{\mc S^1\times\mc S^3}R\big\}\big\{\max_{\mc S^1\times\mc S^3}g^2\big\}>-3$
at $t=0$; and
\item there exists $T<\infty$ such that $\limsup_{t\nearrow T}\max_{s\in\mc S^1}|\Rc(s,t)|=\infty$.
\end{enumerate}
\end{assumption}

\begin{assumption}	\label{StrongerAssumption} 
$\big(\mc S^1\times\mc S^3,G(t)\big)$ is a warped Berger solution of Ricci flow that satisfies
Assumption~\ref{MildAssumption} and has the additional properties that at $t=0$,
\begin{enumerate}
\item $f\geq(1-\ve)g$ for some $\ve$ small enough\,\footnote{Note that $\ve=1/4$ is
sufficiently small.} that
$2(1-\ve)^5+4(1-\ve)^4>4/3$; and
\item $|f_s|\leq1$.
\end{enumerate}
\end{assumption}

\begin{assumption}	\label{Reflection}
$\big(\mc S^1\times\mc S^3,G(t)\big)$ is a warped Berger solution of Ricci flow that satisfies
Assumption~\ref{StrongerAssumption} and is reflection symmetric at $t=0$, with its smallest neck
located at a fixed point $\xi_*\in\mc S^1$.
\end{assumption}

As we observe below, it follows easily from our construction in Appendix~\ref{InitialData}
that these assumptions are not vacuous. Our main results are as follows:

\begin{main}
The eccentricity of every warped Berger solution of Ricci flow is uniformly bounded: there exists
$C_0$ depending only on the initial data such that the estimate
\begin{equation}
\label{f-g}
|f-g|\leq C_0\min\{f,g\}
\end{equation}
holds pointwise for as long as the solution exists, without additional assumptions.
\medskip

\noindent \textsc{(i)} There exist open sets of warped Berger metrics satisfying
Assumption~\ref{MildAssumption} such that all solutions originating in
these sets develop local neckpinch singularities at some $T<\infty$. 
Each such solution has the properties that

\textsc{(a)} the ordering $f\leq g$ is preserved;

\textsc{(b)} the singularity is Type-I, with $|\Rc|\leq C\big(\min f(\cdot,t)\big)^{-2}$, and
\begin{equation}
\nonumber
\label{minf}
\frac{1}{C}\sqrt{T-t}\leq\min f(\cdot,t)\leq C\sqrt{T-t};
\end{equation}

\textsc{(c)} the diameter is bounded as $t\nearrow T$.
\medskip

\noindent \textsc{(ii)}There exist open sets of warped Berger metrics satisfying
Assumption~\ref{StrongerAssumption} such that as solutions originating in
these sets become singular, they become asymptotically round at rates that
break scale invariance. Specifically, in addition to the properties above,
they satisfy the following\,\footnote{The first three estimates, which are proved in
Section~\ref{SharperEstimates}, hold under the weaker assumption that $|f_s|\leq2/\sqrt3$ initially.}
$C^0$, $C^1$, and $C^2$ bounds at the neck:\,\footnote{The $\kappa_{ij}$ here are the sectional
curvatures defined in equations~\eqref{Kurv12}--\eqref{Kurv02} below. Note that for simple warped-product
metrics with $f=g$, one has $\kappa_{12}=\kappa_{23}$ and $\kappa_{01}=\kappa_{02}$.}
\footnote{We arrange these estimates to emphasize the scale invariance of the quantities on the \textsc{lhs}.}
\begin{align}
\label{|f-g|}
(T-t)^{-1/2}|f-g|&\leq C\sqrt{T-t},\\
\label{kappa12}
(T-t)|\kappa_{12}-\kappa_{23}|&\leq C\sqrt{T-t},\\
\label{kappa01}
(T-t)|\kappa_{01}-\kappa_{02}|&\leq C\sqrt{T-t}.
\end{align}
In a neighborhood of each smallest neck, where $\kappa_{01}<0$, there is the further bound
\begin{equation}
\label{kappa0102}
(T-t)\,\big(|\kappa_{01}| + |\kappa_{02}| \big) \leq \frac{C}{|\log(T-t)|}.
\end{equation}
The radius of a smallest neck is $(1+o(1))\,2\sqrt{T-t}$.
Type-I blowups $\tilde G=(T-\nobreak t)^{-1}G$ of the solution converge near each neck to the shrinking
cylinder soliton. If $S$ is the arclength from a smallest neck, and $\sigma:=S/\sqrt{T-t}$,
then there exist constants $c,\,C<\infty$ independent of time, such that 
as $t\nearrow T$, the estimates
\begin{equation}
\label{f/2}
1+o(1)\leq\frac{f}{2\sqrt{T-t}}\leq1+C\,\frac{\sigma^2}{|\log(T-t)|}
\end{equation}
and
\begin{equation}
\label{g/2}
1+o(1)\leq\frac{g}{2\sqrt{T-t}}\leq\big(1+o(1)\big)\left(1+C\,\frac{\sigma^2}{|\log(T-t)|}\right)
\end{equation}
hold for $|\sigma|\leq c\sqrt{|\log(T-t)|}$, and the estimate
\begin{equation}
\label{sigma}
\frac{f}{\sqrt{T-t}}+\frac{g}{\sqrt{T-t}}\leq C\,\frac{|\sigma|}{\sqrt{|\log(T-t)|}}
\sqrt{\log\left(\frac{|\sigma|}{\sqrt{|\log (T-t)|}}\right)}
\end{equation}
holds for $c\sqrt{|\log(T-t)|}\leq|\sigma|\leq(T-t)^{-\ve/2}$, for $\ve\in(0,1)$.
\medskip

\noindent \textsc{(iii)} There exist open sets of reflection-symmetric warped Berger metrics
satisfying Assumption~\ref{Reflection} such that any solution originating in these sets has the
following property: for any small $\delta$ and large $\Sigma$, there exist $T_*<T$
and $C$ such that the stronger estimate
\[
(T-t)^{-1/2}|f-g|\leq C(T-t)^{1+\delta}
\]
holds for all $|\sigma|\leq\Sigma$ and $T_*<t<T$.
\end{main}

We note that warped Berger solutions may also develop global singularities in finite time;
see Remarks~\ref{NotVacuous}--\ref{Local/Nonlocal} below.
We further note that the assumption $f\leq g$ is geometrically natural for initial data giving rise
to neckpinch singularities, in the following sense. Manifolds with $f\gg g$ locally resemble
a product of a small $\mc S^2$ with a large surface and can have substantially
negative scalar curvature. So it is not unreasonable to expect qualitatively different behavior
for solutions originating from such initial data.
\smallskip

Our results in this paper are obtained in a series of Lemmas that prove more than we have
summarized in the Main Theorem. The paper is organized as follows. In Appendix~\ref{Ansatz},
we review basic geometric calculations that show in particular that the metric
\emph{Ansatz}~\eqref{SimpleAnsatz} and the inequality $f\leq g$ are preserved under Ricci flow.
In Section~\ref{floweq}, we summarize the conclusions of Appendix~\ref{Ansatz} that are needed
in the remainder of the paper. In Section~\ref{control}, we first prove estimate \eqref{f-g}, which
requires no assumptions beyond the form~\eqref{SimpleAnsatz} of the metric.
The results in Part \textsc{(i)} of the Main Theorem, which rely only on
Assumption~\ref{MildAssumption}, are proved in the remainder of Section~\ref{control}
and Section~\ref{Singular}. The results in Part \textsc{(ii)} of the Main Theorem,
which rely on Assumption~\ref{StrongerAssumption}, are proved in
Sections~\ref{SharperEstimates}--\ref{cylinder}. The results in Part \textsc{(iii)},
which rely on Assumption~\ref{Reflection}, are proved in Section~\ref{reflection}.
In Appendix~\ref{InitialData}, we demonstrate the existence of sets (open in the subspace
of metrics with prescribed symmetries) of initial data that satisfy our various Assumptions.
Then in Appendix~\ref{FormalNeckpinch},  we study parabolic dilations that motivate
the calculations in Section~\ref{reflection} and lead one to expect that the precise
asymptotics proved in \cite{AK07} for rotationally symmetric neckpinches 
should be satisfied by the the non-rotationally symmetric solutions analyzed here.

\begin{acknowledgment}
The authors warmly thank Peter Gilkey for suggesting a version of the problem studied in this paper.
\end{acknowledgment}

\section{Ricci flow equations for warped Berger metrics}
\label{floweq}
It follows from the calculations in Appendix~\ref{Ansatz} that for metrics of the
form~\eqref{SimpleAnsatz}, the curvatures of the metric induced on each fiber
$\{s\}\times\mc S^3$ are $\up\kappa_{12}=\up\kappa_{31}=f^2/g^4$ and
$\up\kappa_{23}=(4g^2-3f^2)/g^4$. The curvatures of the corresponding vertical planes in the
total space are
\begin{equation}
\label{Kurv12}
 \kappa_{12}=\kappa_{31}=\frac{f^2}{g^4}-\frac{f_s g_s}{fg}
\end{equation}
and
\begin{equation}
\label{Kurv23}
 \kappa_{23}=\frac{4g^2-3f^2}{g^4}-\frac{g_s^2}{g^2}.
\end{equation}
The curvatures of mixed vertical-horizontal planes in the total space are
\begin{equation}
\label{Kurv01}
\kappa_{01}=-\frac{f_{ss}}{f}
\end{equation}
and
\begin{equation}
\label{Kurv02}
\kappa_{02}=\kappa_{03}=-\frac{g_{ss}}{g}.
\end{equation}

Using \eqref{Kurv12}--\eqref{Kurv02} together with \eqref{RFsystem}, one determines
that the Ricci flow equations for these geometries take the form 
\begin{subequations}		\label{SRFsystem}
\begin{align}
f_t&=f_{ss}+2\frac{g_s}{g}f_s-2\frac{f^3}{g^4},\\
g_t&=g_{ss}+\left(\frac{f_s}{f}+\frac{g_s}{g}\right)g_s
+2\frac{f^2-2g^2}{g^3}.
\end{align}
\end{subequations}
If $f=g$, this system reduces to equation~(10) in \cite{AK04},
with $n=3$ and $\psi=f$.

To obtain this strictly parabolic form \eqref{SRFsystem} for the Ricci flow equations, 
 we have fixed a gauge,\footnote{By equation~\eqref{DistanceEvolution} in
Appendix~\ref{Ansatz}, the gauge function $\rho:=\frac{\partial s}{\partial\xi}$
evolves by $(\log\rho)_t=\frac{f_{ss}}{f}+2\frac{g_{ss}}{g}$. If $f=g$, this
evolution equation reduces to equation~(11) in \cite{AK04}, with $\vp=\rho$.}
replacing the non-geometric coordinate $\xi\in\mc S^1$ with a coordinate
$s(\xi,t)$ representing arclength
from a fixed but arbitrary point $\xi_0\in\mc S^1$.
By a variant of Calabi's trick, we may always assume that $s$ is a smooth coordinate at any
spatial point where we apply the maximum principle. 
We note that this choice of gauge results in the commutator formula
\begin{equation}	\label{commute}
    \left[\frac{\partial}{\partial t},\frac{\partial}{\partial s}\right]
    =-(\log\rho)_t\frac{\partial}{\partial s}
    =-\left(\frac{f_{ss}}{f}+2\frac{g_{ss}}{g}\right)\frac{\partial}{\partial s}.
\end{equation}

\begin{remark}
The system \eqref{SRFsystem}  can be re-expressed in the more geometric form 
\begin{align*}
(\log f)_t&=-\kappa_{01}-2\kappa_{12},\\
(\log g)_t&=-\kappa_{02}-\kappa_{23}-\kappa_{31}.
\end{align*}
\end{remark}

\section{Controlling the evolving geometries}
\label{control}
To proceed, we derive various evolution equations implied by the Ricci flow
system~\eqref{SRFsystem}. In doing so, we  use the fact that for
any $C^2$ function $\phi(s)$, one has
\begin{equation}	\label{DefineLaplacian}
 \Delta\phi=\phi_{ss}+\left\{\frac{f_s}{f}+2\frac{g_s}{g}\right\}\phi_s.
\end{equation}

\subsection{The shape of the metric}
Because $\mc S^1\times\mc S^3$ is compact, there are well-defined functions $M$
and $\dn M$ given by
\begin{equation}
M(s,t):=\min\{f(s,t),\,g(s,t)\}\qquad\text{and}\qquad\dn
M(t):=\min_{s\in\mc S^1}M(s,t).
\end{equation}
One readily verifies that $\dn M$ is a Lipschitz continuous function of time.\footnote{Comments
related to this verification appear in the proof of Lemma~\ref{lem-upper}.}
\medskip

We begin by considering scale-invariant quantities $(f-g)/g$ and $(g-f)/f$ that measure
eccentricity: how far each fiber $\{s\}\times\mc S^3$ is from being round. The evolution of
these quantities is governed by the equations
\begin{equation}
\label{Evolvf-g}
 \left(\frac{f-g}{g}\right)_t=\Delta\left(\frac{f-g}{g}\right)
 +\left(\frac{g_s}{g}-\frac{f_s}{f}\right)\left(\frac{f-g}{g}\right)_s
 -4\,\frac{f}{g^3}\left(\frac{f+g}{g}\right)\left(\frac{f-g}{g}\right)
\end{equation}
and
\begin{equation}
\label{Evolvg-f}
  \left(\frac{g-f}{f}\right)_t=\Delta\left(\frac{g-f}{f}\right)
  +\left(\frac{f_s}{f}-\frac{g_s}{g}\right)\left(\frac{g-f}{f}\right)_s
  -4\,\frac{f}{g^3}\left(\frac{g+f}{f}\right)\left(\frac{g-f}{f}\right),
\end{equation}
respectively. Using these equations, we show that 
 the fibers must become
round near any points where $f$ or $g$ become zero, as expressed in the following Lemma.

\begin{lemma}	\label{C0pinching}
There exists $C_0$ depending only on the initial data such that the estimate
\[
 |f-g|\leq C_0M
\]
holds for as long as a given solution exists.
\end{lemma}
\begin{proof}
Applying the parabolic maximum principle to the evolution equations~\eqref{Evolvf-g} and
\eqref{Evolvg-f} for  $(f-g)/g$ and $(g-f)/f$, one obtains $C_0$ depending only on the initial data such that
\[
\left|\frac{f-g}{g}\right|\leq C_0
\qquad\text{and}\qquad
\left|\frac{g-f}{f}\right|\leq C_0
\]
for as long as a solution exists. The result immediately follows.
\end{proof}

We next derive a two-sided time-dependent bound for $\dn M$, starting with an upper bound.

\begin{lemma}
\label{lem-upper}
If there exists $T<\infty$ such that $\dn M(T)=0$, then
there exists a uniform constant $C$ such that
\[
\dn M^2\leq C(T-t).
\]
\end{lemma}

\begin{proof}
The Sturmian theorem \cite{Ang88} applied to $f$ implies that for all but a finite set of
times, $f$ is a Morse function with smoothly evolving critical points, whence it follows that
the function $\dn f(t):=\min\{f(s,t):f_s(s,t)=0\}$ is Lipschitz continuous.
We now slightly abuse notation by regarding $f$ as $f(\xi,t)$, where the spatial coordinate $\xi$ is
independent of time --- i.e., we ignore here the arclength coordinate $s(\xi,t)$.
If $t$ is such that $\dn f'(t)$ exists, then it follows from  the implicit function theorem that  there
exists a function $\bar\xi(\bar t)$ defined for all $\bar t$ in a sufficiently small neighborhood of $t$
such that $f_{\bar\xi}(\bar\xi(\bar t),\bar t)=0$. Therefore,  one has
\begin{align*}
\frac{d}{dt}\dn f(t)
	&=\frac{\partial}{\partial t}f(\bar\xi(t),t)
	+\frac{\partial}{\partial\bar\xi}f(\bar\xi(t),t)\frac{d\bar\xi}{dt}\\
	&=f_t(\bar\xi(t),t)\\
	&=f_{ss}(\bar\xi(t),t)-2\frac{f^3(\bar\xi(t),t)}{g^4(\bar\xi(t),t)}\\
	&\geq-\frac{C}{\dn f(t)},
\end{align*}
since it follows from Lemma~\ref{C0pinching} that $f$ and $g$ are comparable,
and since $f$ attains a local minimum in space at $\bar\xi(t)$. Thus  there exists a uniform
constant $C$ such that $\frac{d}{dt}\{(\dn f)^2\}\geq-C$.

An entirely analogous argument applies to $\dn g(t):=\min\{g(s,t):g_s(s,t)=0\}$.
It follows easily that $\frac{d}{dt}(\dn M^2)\geq-C$ holds almost everywhere in
time,\footnote{This differential inequality may be interpreted as the $\limsup$ of
forward difference quotients.} whereupon integration yields
\[
-\dn M^2(t)=\dn M^2(T)-\dn M^2(t)\geq-C(T-t).
\]
\end{proof}

\begin{lemma}
\label{lem-lower}
Suppose that at time $t=0$, the metric satisfies $f\leq g$, and the scalar curvature satisfies
$\big\{\min_{\mc S^1\times\mc S^3}R\big\}\big\{\max_{\mc S^1\times\mc S^3}g^2\big\}>-3$.
If there exists $T<\infty$ such that $\dn M(T)=0$, then there exists a uniform constant $c$ such that
\[
\dn M^2\ge c(T-t).
\]
\end{lemma}

\begin{proof}
The positive function $m(t) := \min_{\mathcal{S}^1\times\mathcal{S}^3}(fg^2)$ is Lipschitz continuous.
It follows from \eqref{SRFsystem} that
\[
\frac{\partial}{\partial t} \log(fg^2) =
\left(\frac{f_{ss}}{f} + 2\frac{g_{ss}}{g}\right) + 4\frac{f_s g_s}{fg} + 2\frac{g_s^2}{g^2} + 2 \frac{f^2}{g^4} - \frac{8}{g^2}.
\]
Since $R$ is a supersolution of the heat equation (in the sense that $(\partial_t-\Delta)R\geq0$), there
exists a constant $r_0$ depending only on the initial data such that for as long as the flow exists, one has
\[
r_0\leq R=\kappa_{01}+\kappa_{02}+\kappa_{03}+\kappa_{12}+\kappa_{23}+\kappa_{31}.
\]
Substituting in expressions \eqref{Kurv12}--\eqref{Kurv02} for the curvatures and simplifying, one obtains
\begin{equation}
\label{eq-useful}
gf_{ss}+2fg_{ss}\leq(4-r_0 g^2)\frac{f}{g}-\frac{f^3}{g^3}-2f_s g_s-\frac{f}{g}g_s^2.
\end{equation}
Using this estimate and the consequence of Lemma~\ref{Berger} (in Appendix~\ref{Ansatz})
that the ordering $f \le g$ is preserved along the flow, we obtain
\begin{align*}
\frac{\partial}{\partial t} \log(fg^2)
&\le\left(\frac{4-r_0 g^2}{g^2} - \frac{f^2}{g^4} - 2\frac{f_s g_s}{fg} - \frac{g_s^2}{g^2}\right)
	+ 4\frac{f_s g_s}{fg} + 2\frac{g_s^2}{g^2} + 2 \frac{f^2}{g^4} - \frac{8}{g^2} \\
&= \frac{4-r_0 g^2}{g^2} + \frac{f^2}{g^4} + 2\,\frac{f_s g_s}{f g} + \frac{g_s^2}{g^2} - \frac{8}{g^2} \\
&\le -\frac{3+r_0 g^2}{g^2} + 2 \frac{g_s}{g}(\log(fg^2))_s - \frac{3g_s^2}{g^2}.
\end{align*}
This implies that almost everywhere in time, one has
\[
\frac{d}{dt}(\log m) \le -r_0-\frac{3}{g^2}.
\]
It is easy to see from \eqref{SRFsystem} that if $f\leq g$, then $g_{\max}(\cdot,t)$ is a
non-increasing function of time. So it follows from our assumptions on $f$, $g$, and $R$
that there exists $c_0>0$ such that $r_0\geq-(3-c_0)/g^2$, which implies that
\[\frac{dm}{dt} \le -c_0f \le -c (fg^2)^{\frac 13},\]
where $c>0$ is another uniform constant whose existence follows from Lemma~\ref{C0pinching}.
Because there exists $T<\infty$ with $\dn M(T) = 0$, it is clear that $m(T)=0$. Integrating the
a.e.~inequality
\[
\frac{dm}{dt} \le -c m^{\frac 13}
\]
over the time interval $[t,T]$, we thus obtain
\[
m(t)^{\frac 23} \ge c\, (T - t).
\]
Now by Lemma~\ref{C0pinching}, the inequality
\[
f^3\geq\frac{fg^2}{C}\geq\frac{\min_{\mc S^1\times\mc S^3}(fg^2)}{C}=\frac{m(t)}{C}
\]
holds everywhere in space and time, which implies in particular that\,\footnote{Here as
elsewhere in this paper, we follow the convention in analysis that uniform constants
are allowed to change from line to line without relabeling.}
\[
\min f(\cdot,t)^2\geq\frac{m(t)^{\frac23}}{C}\geq\frac{c}{C}(T-t).
\]
The same reasoning applies to $\min g(\cdot,t)^2$, whence the result follows.
\end{proof}

\subsection{Evolution of first derivatives}
Using~\eqref{commute}, it is straightforward to compute that
\begin{equation}	\label{f_s-evolution}
(f_s)_t=\Delta(f_s)-2\frac{f_s}{f}(f_s)_s
 	-\left\{6\frac{f^2}{g^4}+2\frac{g_s^2}{g^2}\right\}f_s
	+8\frac{f^3}{g^5}\,g_s
\end{equation}
and
\begin{equation}	\label{g_s-evolution}
 (g_s)_t = \Delta(g_s)-2\frac{g_s}{g}(g_s)_s
 +\left\{\frac{4}{g^2} - \frac{g_s^2}{g^2} - \frac{f_s^2}{f^2} - 6\frac{f^2}{g^4}\right\} g_s
 + 4\frac{f}{g^3} f_s.
\end{equation}
If $f=g$, these reduce to equation~(16) in \cite{AK04}.

\begin{lemma}
\label{lem-der}
Suppose that $f \leq g$ at time $t=0$, and define
\begin{align*}
C_f &:= \max\left\{ \frac{2}{\sqrt3},\;\max |f_s(\cdot,0)|\right\},\\
C_g &:= \max\left\{ 2\sqrt2,\;\max |g_s(\cdot,0)|\right\}.
\end{align*}
Then for as long as a solution exists, one has
\[
|f_s| \le C_f \qquad\mbox{and}\qquad   |g_s| \le C_g.
\]
\end{lemma}

\begin{proof}
Consider $(f_s)_{\max}$, and assume that $(f_s)_{\max} \ge C > 0$ where $C$ is sufficiently large.
By Lemmas~\ref{C0pinching} and \ref{Berger} (in Appendix~\ref{Ansatz}), we have 
\begin{equation}
\label{eq-double-ineq}
f(\cdot,t) \le g(\cdot,t)\leq(1+C_0)f(\cdot,t)
\end{equation}  
for as long as a solution exists, where $C_0$ is the uniform constant in Lemma~\ref{C0pinching}.
Recalling the evolution equation~\eqref{f_s-evolution} for $f_s$ and
applying weighted Cauchy--Schwarz,\footnote{To wit, we estimate
$ab\leq\epsilon a^2+\frac{1}{4\epsilon}b^2$, with $a=\frac{g_s}{g}$, $b=\frac{f^3}{g^4}$, and
$\epsilon=\frac{1}{2\sqrt3}$.} we obtain
\begin{align*}
\frac{d}{dt} (f_s)_{\max}
&\le -(f_s)_{\max}\, \left(6 \frac{f^2}{g^4} + 2 \frac{g_s^2}{g^2}\right) + 8 \frac{f^3}{g^5} g_s \\
&\le-C\, \left(6 \frac{f^2}{g^4} + 2 \frac{g_s^2}{g^2}\right)  + \frac{4}{\sqrt{3}}\frac{g_s^2}{g^2}
+ \frac{12}{\sqrt{3}}\frac{f^6}{g^8} \\
&\le -C\, \left(6 \frac{f^2}{g^4} + 2 \frac{g_s^2}{g^2}\right)  + \frac{4}{\sqrt{3}}\frac{g_s^2}{g^2}
+ \frac{12}{\sqrt{3}}\frac{f^2}{g^4} \\
&\le 0,
\end{align*}
if we choose $C \ge 2/\sqrt3$. This implies that there is a sufficiently large constant $C$ such that
$(f_s)_{\max} \le C$ uniformly, as long as the flow exists. Similarly we also get a uniform bound
$(f_s)_{\min} \ge -C$.
\smallskip

We now consider $(g_s)_{\max}$. Suppose $(g_s)_{\max}\geq C_g\geq\sqrt8$. Then
\[
\frac{4}{g^2}-\frac{g_s^2}{g^2}\leq-\frac12\,\frac{g_s^2}{g^2}\leq-\frac{4}{g^2}.
\]
So it follows from the evolution equation~\eqref{g_s-evolution} for $g_s$ that
\[
\frac{d}{dt}(g_s)_{\max}\leq
-\left\{\frac{4}{g^2}+\frac{f_s^2}{f^2}\right\}(g_s)_{\max}+\frac{4ff_s}{g^3}.
\]
Using weighted Cauchy--Schwarz and the fact that the inequality $f\leq g$ is
preserved, we obtain
\[
\frac{4ff_s}{g^3}\leq\frac{4f^4}{g^6}+\frac{f_s^2}{f^2}\leq\frac{4}{g^2}+\frac{f_s^2}{f^2}.
\]
Hence at any sufficiently large value of $(g_s)_{\max}$, one has
\[
\frac{d}{dt}(g_s)_{\max}\leq0.
\]
A similar argument shows that $(g_s)_{\min}\geq-C_g$.
\end{proof}

\begin{corollary}	\label{BoundVertical}
If $f\leq g$ initially, then there exists $C$
depending only on the initial data such that the estimate
\[ |\kappa_{12}| + |\kappa_{31}| + |\kappa_{23}| \le \frac{C}{M^2}\]
holds for as long as a solution exists.
\end{corollary}

\begin{proof}
Because
\[\kappa_{12} = \kappa_{31} = \frac{f^2}{g^4} - \frac{f_s g_s}{f g} \qquad \mbox{and} \qquad
\kappa_{23} = \frac{4}{g^2} - \frac{3f^2}{g^4} - \frac{g_s^2}{g^2},\]
the stated bound follows immediately from Lemma~\ref{C0pinching} and Lemma~\ref{lem-der}.
\end{proof}

\subsection{Evolution of second derivatives}
After further tedious but straightforward computations, one finds that $\kappa_{01}$ and
$\kappa_{02}$ evolve by
\begin{align}	\label{kappa01evolution}
 (\kappa_{01})_t&=\Delta(\kappa_{01})+2\kappa_{01}^2
 -4\left\{\frac{g_s^2}{g^2}+\frac{f^2}{g^4}\right\}\kappa_{01}
 +4\left\{\kappa_{12}+\frac{f^2}{g^4}\right\}\kappa_{02}\\
 &\qquad+12\frac{f_s^2}{g^4}+40\frac{f^2g_s^2}{g^6}-48\frac{ff_s g_s}{g^5}
 -4\frac{f_s g_s^3}{fg^3}\notag
\end{align}
and
\begin{align}	\label{kappa02evolution}
 (\kappa_{02})_t&=\Delta(\kappa_{02})+2\kappa_{02}^2 \\
&\qquad+ \left(\frac{4f^2}{g^4} - \frac{2f_s g_s}{f g}\right)\, \kappa_{01} 
+ \left(\frac{8}{g^2} - \frac{8f^2}{g^4} - \frac{2f_s^2}{f^2} - \frac{4g_s^2}{g^2}\right)\, \kappa_{02}\notag\\
&\qquad- \frac{4f_s^2}{g^4} + \frac{24 f f_s g_s}{g^5} - \frac{2f_s^3 g_s}{f^3 g} - \frac{24 f^2 g_s^2}{g^6} + \frac{8g_s^2}{g^4} - \frac{2g_s^4}{g^4},\notag
\end{align}
respectively. If $f=g$, these reduce to equation~(22) in \cite{AK04}, using the
identifications $K=-\kappa_{01}=-\kappa_{02}$ and $L=\kappa_{12}=\kappa_{23}$.

\medskip

It follows from  \cite{Ses05} that a singularity occurs at $T<\infty$ only if 
\[
\limsup_{t\nearrow T}\max_{s\in\mc S^1}|\Rc(s,t)\,|=\infty.
\]
We now show that all remaining curvatures are controlled by
$\dn M$ at a finite-time singularity.

\begin{lemma}	\label{BoundMixed}
Suppose that at time $t=0$, the metric satisfies $f\leq g$, and the scalar curvature satisfies
$\big\{\min_{\mc S^1\times\mc S^3}R\big\}\big\{\max_{\mc S^1\times\mc S^3}g^2\big\}>-3$.
If the norm of $\Rc$ becomes unbounded as $t\nearrow T<\infty$,  then $\dn M(T)=0$,
and there exists a uniform constant $C$ such that
\[
|\kappa_{01}| + |\kappa_{02}| + |\kappa_{03}| \leq\frac{C}{\dn M^2}.
\]
\end{lemma}

\begin{proof}
Corollary~\ref{BoundVertical} bounds the sectional curvatures of vertical planes  by $C/M^2$.
So it remains only to consider the mixed curvatures $\kappa_{01}$ and $\kappa_{02}=\kappa_{03}$.

\smallskip

To control $\kappa_{01}$ from above, we work with
$K^*:=\kappa_{01}+a\,f_s^2/f^2+b\,g_s^2/g^2$,
where $a$ and $b$ are positive constants to be chosen. 
Clearly, it follows from this definition that $K^*$ is an upper bound for $\kappa_{01}$.
To derive an estimate for the time derivative of $K^*$, we begin by estimating
the time derivative of $\kappa_{01}$.  Applying  Lemma~\ref{C0pinching}, Lemma~\ref{lem-der},
and the Cauchy--Schwarz inequality to equation~\eqref{kappa01evolution}, we obtain
\begin{align*}
(\kappa_{01})_t&\leq\Delta(\kappa_{01})+2\kappa_{01}^2
+C\left\{\frac{\kappa_{01}}{M^2}+\frac{\kappa_{02}}{M^2}+\frac{1}{M^4}\right\}\\
&\leq\Delta(\kappa_{01})+C\left\{\kappa_{01}^2+\kappa_{02}^2+\frac{1}{M^4}\right\}.
\end{align*}
We next calculate the time derivatives of the quadratic terms in $K^*$, obtaining 
\begin{align*}
\left(\frac{f_s^2}{f^2}\right)_t
&= \Delta \left(\frac{f_s^2}{f^2}\right) - \frac{8f_s^2}{g^4} - \frac{4f_s^2}{f^4}
+ \frac{16 f f_s g_s}{g^5} - \frac{4 f_s^2 g_s^2}{f^2 g^2} - \frac{4 f_s^2}{f^2} \kappa_{01} - 2\kappa_{01}^2, \\
\left(\frac{g_s^2}{g^2}\right)_t
&= \Delta\left(\frac{g_s^2}{g^2}\right) + \frac{8 f f_s g_s}{g^5} 
- \frac{16f^2 g_s^2}{g^6} + \frac{16 g_s^2}{g^4}- \frac{2f_s^2 g_s^2}{f^2 g^2} - \frac{6g_s^4}{g^4} - \frac{4g_s^2}{g^2}\kappa_{02} - 2\kappa_{02}^2.
\end{align*}
Again applying Lemma~\ref{C0pinching}, Lemma~\ref{lem-der}, and weighted 
Cauchy--Schwarz, we get
\begin{align*}
\left(\frac{f_s^2}{f^2}\right)_t
&\leq\Delta \left(\frac{f_s^2}{f^2}\right) -\kappa_{01}^2+\frac{C}{M^4},\\
\left(\frac{g_s^2}{g^2}\right)_t
&\leq \Delta\left(\frac{g_s^2}{g^2}\right)-\kappa_{02}^2+\frac{C}{M^4}.
\end{align*}
It immediately follows that for $a,b$ chosen large enough, one has
\[
K^*_t\leq\Delta K^*+\frac{C}{M^4},
\]
and hence
\begin{equation}
\label{ddtK*max}
\frac{d}{dt}(K^*)_{\max}\leq\frac{C}{\dn M^4}.
\end{equation}
This inequality, together with the mean value theorem, imply   that $K^*$ and hence $\kappa_{01}$
cannot approach $+\infty$ on any time interval on which $\dn M$ is bounded
away from zero.

To control $\kappa_{01}$ from below, we work with $K_*:=\kappa_{01}-c\,g_s^2/g^2$,
where $c$ is a positive constant to be chosen. This quantity clearly serves as a lower
bound for $\kappa_{01}$. Calculating as above, we obtain the estimate
\[
(\kappa_{01})_t\geq\Delta(\kappa_{01})-C\left\{\kappa_{02}^2+\frac{1}{M^4}\right\}.
\]
Combining this with the inequality derived above for $\,g_s^2/g^2$, we see that 
for $c$ chosen large enough, one has $(K_*)_t\geq\Delta K_*-C/M^4$,
and hence
\begin{equation}
\label{ddtK*min}
\frac{d}{dt}(K_*)_{\min} \geq-\frac{C}{\dn M^4}.
\end{equation}
It follows that $K_*$, and hence $\kappa_{01}$, cannot approach $-\infty$ on any
time interval on which $\dn M$ is bounded away from zero. Combining this result with
that obtained above, we see that $\kappa_{01}$ becomes singular at a finite time $T$ only if  $\dn M=0.$ 

To determine the specific relation between $\kappa_{01}$ and $\dn M$, we combine the estimates
for $\dn M$ obtained in Lemma ~\ref{lem-upper} and Lemma~\ref{lem-lower} with 
estimates~\eqref{ddtK*max} and \eqref{ddtK*min}, thereby obtaining 
\[
\frac{d}{dt}K^*_{\max}\leq\frac{C}{(T-t)^2}\qquad\mbox{and}\qquad
\frac{d}{dt}(K_*)_{\min}\geq-\frac{C}{(T-t)^2}.
\]
Integrating these inequalities leads to the estimate $|\kappa_{01}|\leq C_1+C/(T-t)$.
Then applying Lemma~\ref{lem-upper} again, we get the desired control  on $\kappa_{01}$, which is
\[
|\kappa_{01}|\leq\frac{C}{\dn M^2}.
\]

\smallskip
The estimate for $|\kappa_{02}|$ is obtained similarly, using
$\kappa_{02}+a\,f_s^2/f^2+b\,g_s^2/g^2$ for an upper bound,
and $\kappa_{02} -c\,f_s^2/f^2$ for a lower bound.
\end{proof}

\section{Analysis of singularities}	\label{Singular}
In this section, we study solutions of Ricci flow satisfying
Assumption~\ref{MildAssumption}, as stated in the introduction.

\begin{remark}	\label{NotVacuous}
To see that Assumption~\ref{MildAssumption} is not vacuous, it suffices to observe that initial data
with $f\leq g$ both constant have strictly positive constant scalar curvature $R=(4g^2-f^2)/g^4$.
So there is a neighborhood of these products in the space of metrics $\mathfrak{Met}(\mc S^1\times\mc S^3)$
such that all warped Berger solutions originating from initial data in this neighborhood satisfy
the first two hypotheses of the Assumption and also become singular in finite time. The last fact
follows from the standard estimate $R_{\max}(t)\geq\big( [R_{\max}(0)]^{-1}-t/2 \big)^{-1}$. By
Lemma~\ref{C0pinching}, these solutions will develop finite-time (global or local) singularities
and will satisfy $f=g$ (hence ``become round")  at all points where $M=\dn M=0$.
\end{remark}

\begin{remark}	\label{Local/Nonlocal}
It is expected that open sets of warped Berger solutions will encounter global singularities in
which the geometry shrinks uniformly around the $\mc S^1$ factor; for example, this is expected
for solutions originating from initial data sufficiently near the products described in Remark~\ref{NotVacuous}.
On the other hand, we show in Appendix~\ref{InitialData} below that there exist open sets of warped
Berger solutions that develop local neckpinch singularities. Unless otherwise stated, the results
in this paper apply to both cases.
\end{remark}

It is clear from Corollary~\ref{BoundVertical} and the proof of Lemma~\ref{BoundMixed} that
the singular set $\Sigma$, i.e.~the set of points $\{\xi\}\times\mc S^3\subseteq\mc S^1\times\mc S^3$
such that $\limsup_{t\nearrow T}|\Rc(\xi,t)\,|=\infty$, coincides with the set $\Sigma_0$ of points such
that $M(\xi,t)\searrow0$ as $t\nearrow T$. Moreover, Lemma~\ref{lem-lower} shows that the singularity
is Type-I. It therefore follows from  \cite{EMT11} that  $\Sigma=\Sigma_R$, where $\Sigma_R$ denotes
the set of points at which the scalar curvature blows up at the Type-I rate as $t\nearrow T$.
\medskip

Our first observation is that the solution has a well-defined profile at the singular time.

\begin{lemma}	\label{UniformlyLipschitz}
If  a solution $\big(\mc S^1\times\mc S^3,G(t)\big)$ of Ricci flow satisfies Assumption~\ref{MildAssumption}
and becomes singular at $T<\infty$, then the limits $\lim_{t\nearrow T}f(\xi,t)$ and $\lim_{t\nearrow T}g(\xi,t)$
both exist for all $\xi\in\mc S^1$.
\end{lemma}

\begin{proof}
We observe that
\[
(f^2)_t=2ff_{ss}+4\frac{f}{g}f_s g_s-4\frac{f^4}{g^4}
\]
and
\[
(g^2)_t=2gg_{ss}+2\left(\frac{g}{f}f_s g_s+g_s^2\right)+4\frac{f^2}{g^2}-8.
\]
It thus follows from Lemmas~\ref{C0pinching}, \ref{lem-der}, and \ref{BoundMixed} that there is
a uniform constant $C$ such that
\[
|(f^2)_t|\leq C\qquad\mbox{and}\qquad|(g^2)_t|\leq C.
\]
Consequently both $f^2$ and $g^2$ are uniformly Lipschitz-continuous functions of time.
\end{proof}

\begin{corollary}	
\label{NotTooFast}
If a solution $\big(\mc S^1\times\mc S^3,G(t)\big)$ of Ricci flow satisfies Assumption~\ref{MildAssumption}
and becomes singular at $T<\infty$,  then there exists a uniform constant $C$ such that
\[
f^2(s,t)\geq f^2(s,0)-Ct\qquad\mbox{and}\qquad
g^2(s,t)\geq g^2(s,0)-Ct.
\]
\end{corollary}

Our next observation concerns the diameter of a Ricci flow solution that satisfies
Assumption~\ref{MildAssumption} but does \emph{not} encounter a global singularity.
(We construct initial data leading to such solutions in Appendix~\ref{InitialData}.) Here we
prove a diameter bound for such solutions.
\begin{lemma}	
\label{DiameterBound}
Suppose a solution $\big(\mc S^1\times\mc S^3,G(t)\big)$ of Ricci flow satisfies
Assumption~\ref{MildAssumption} and becomes singular at $T<\infty$.
If the singular set $\Sigma\neq\mc S^1\times\mc S^3$, then the diameter of the solution remains
bounded as $t\nearrow T$.
\end{lemma}

\begin{proof}
Let $\xi_1,\xi_2\in\mc S^1\times\mc S^3\setminus\Sigma$ be arbitrary points. 
Let $s_1=s(\xi_1,t)$ and $s_2=s(\xi_2,t)$.
The quantity $\rho=ds/d\xi$ evolves by equation~\eqref{DistanceEvolution} (derived in Appendix~\ref{Ansatz}).
Using this along with formulas \eqref{Kurv01}--\eqref{Kurv02} for the curvatures and integrating by parts,
we calculate that
\begin{align*}
\frac{d}{dt}{\rm dist}_{G(t)}(\xi_1,\xi_2)
&=\frac{d}{dt}\int_{\xi_1}^{\xi_2}\rho(\xi,t)\,\mr d\xi\\
&=\int_{s_1}^{s_2}\left(\frac{f_{ss}}{f}+2\frac{g_{ss}}{g}\right)\mr ds\\
&=\left.\left(\frac{f_s}{f}+2\frac{g_s}{g}\right)\right|_{s_1}^{s_2}
+\int_{s_1}^{s_2}\left(\frac{f_s^2}{f^2}+2\frac{g_s^2}{g^2}\right)\mr ds.
\end{align*}
Since $\xi_1,\xi_2\notin\Sigma$, it follows from Lemma~\ref{UniformlyLipschitz}
that there exists $c>0$ such that $f(\xi_i,T)\geq c$ and $g(\xi_i,T)\geq c$ for $i=1,2$.
Combining this with the derivative bounds of Lemma~\ref{lem-der}, we obtain the bound
\[
\left|\left.\left(\frac{f_s}{f}+2\frac{g_s}{g}\right)\right|_{s_1}^{s_2}\right|\leq C.
\]
Again using Lemma~\ref{lem-der}, we bound the integral above by
\begin{align*}
\int_{s_1}^{s_2}\left(\frac{f_s^2}{f^2}+2\frac{g_s^2}{g^2}\right)\mr ds
&\leq C\int_{s_1}^{s_2}\left(\frac{|f_s|}{f^2}+\frac{|g_s|}{g^2}\right)\mr ds\\
&\leq C' C\left(\frac{1}{\min\{f(\cdot,t),\,g(\cdot,t)\}}-\frac{1}{\max\{f(\cdot,t),\,g(\cdot,t)\}}\right).
\end{align*}
To get the final estimate above, we work separately on each interval on which $f$ or $g$
is monotone increasing/decreasing. Because the Sturmian theorem \cite{Ang88}
applied to $f$ and $g$ implies that the number of their critical points cannot increase with time,
we have an \emph{a priori} bound on the number $C'$ of such intervals.

Combining the estimates above and using Lemma~\ref{lem-lower} shows that
\[
\left|\frac{d}{dt}{\rm dist}_{G(t)}(\xi_1,\xi_2)\right|\leq\frac{C}{\dn M}\leq\frac{C}{\sqrt{T-t}},
\]
which is integrable.
\end{proof}

\section{Sharper estimates}	\label{SharperEstimates}

In this section, we obtain stronger results under the more restrictive hypotheses on the initial
data detailed in Assumption~\ref{StrongerAssumption} in the introduction, with the goal of
breaking scaling invariance. It follows easily from Remark~\ref{NotVacuous} above that
Assumption~\ref{StrongerAssumption} is not vacuous. Furthermore, Remark~\ref{WeCanDoIt}
below shows that the assumption is satisfied by an open set of warped Berger solutions that
develop local singularities.
\smallskip


Our first result shows that solutions originating from original data that are not
too far from round become asymptotically round near their singular sets at a
rate that breaks scale invariance, hence that improves upon the scale-invariant
$C^0$ estimate of Lemma~\ref{C0pinching}.

\begin{lemma}
\label{lem-sharper-est}
If a solution $\big(\mc S^1\times\mc S^3,G(t)\big)$ satisfies Assumption~\ref{StrongerAssumption},
then there exists a uniform constant $C$ such that for as long as the flow exists, one has
$0<\frac{1}{f}-\frac{1}{g}\leq C$ and hence
\begin{equation}
\label{g-festim}
0<g - f \le C M^2.
\end{equation} 
\end{lemma}

\begin{proof}
Define $h := \frac{1}{f} - \frac{1}{g}$. It easily follows from \eqref{SRFsystem}
that $h$ evolves by
\begin{equation}
\label{timederivh}
h_t = \Delta h + (f - g) \left(\frac{2f + 4g}{g^5} - \frac{f_s^2}{f^4}\right).
\end{equation}
By Assumption~\ref{StrongerAssumption} and Lemma~\ref{C0pinching},
the inequality $g\leq(1+C_0)f$ is preserved, where $1+C_0=(1-\ve)^{-1}$.
By Assumption~\ref{StrongerAssumption} and Lemma~\ref{lem-der},
the inequality $|f_s|\leq2/\sqrt3$ persists as well. Thus it follows from
our choice of $\ve$ in Assumption~\ref{StrongerAssumption} that
\[
\frac{2f^5 + 4f^4g - f_s^2g^5}{f^4 g^5} \geq
\frac{2(1-\ve)^5+4(1-\ve^4)-\frac43}{f^4} > 0.
\]
Combining this inequality and Lemma~\ref{Berger} (which guarantees that 
$f \le g$ for as long as the flow exists) with evolution equation \eqref{timederivh}, we obtain
\[
\frac{d}{dt}\,h_{\max} \le 0.
\]
It follows that  $0<g-f\leq Cfg$. This inequality and Lemma~\ref{C0pinching} together imply~\eqref{g-festim}.
\end{proof}

We next obtain a $C^1$ estimate for solutions satisfying Assumption~\ref{StrongerAssumption}.
This estimate  improves upon Lemma~\ref{lem-der} and shows that solutions become round in spatial 
neighborhoods of their singular sets.

\begin{lemma}
\label{lem-sharper-der}
If a solution $\big(\mc S^1\times\mc S^3,G(t)\big)$ satisfies Assumption~\ref{StrongerAssumption},
then there exists a uniform constant $C$ such that for as long as the flow exists, one has
\begin{equation}
\label{f-gderiv}
|(f-g)_s| \le CM.
\end{equation}
\end{lemma}

\begin{proof}

Consider the quantity
\begin{equation}
\label{Q}
Q := \left(\frac{f_s}{f} - \frac{g_s}{g}\right)^2.
\end{equation}

We claim that if one can show that $Q \le C$ for some uniform constant $C$, then
estimate~\eqref{f-gderiv} follows.  To verify this claim, we observe that if $Q\leq C$, then one has
\begin{align*}
\left(\frac{f_s - g_s}{f}\right)^2
&=\left\{\frac{f_s}{f}-\frac{g_s}{g}
	+g_s\left(\frac{1}{g}-\frac{1}{f}\right)\right\}^2\\
&\leq 2Q +2|g_s|^2\left(\frac{1}{f} - \frac{1}{g}\right)^2\\
&\leq C.
\end{align*}
Here we have used Lemmas~\ref{lem-der} and \ref{lem-sharper-est} to bound the second term
on the second line. This implies  the result we want in the form $|f_s - g_s|^2 \le Cf^2=CM^2$.

We proceed to prove $Q\leq C$. We readily verify  that $Q$ evolves by
\[
\frac{\partial}{\partial t} Q = \Delta Q - 2Q_s^2 
- 2Q\left(\frac{g_s^2}{g^2} + \frac{f_s^2}{f^2} + \frac{8f^2}{g^4}\right)
+ \frac{16 g_s\, (f^2-g^2)}{g^5}\,\left(\frac{f_s}{f} - \frac{g_s}{g}\right).
\]
We obtain
\begin{align*}
\frac{d}{dt}\,Q_{\max}
&\leq-\frac{16(1-\ve)^2}{f^2}\,Q_{\max} + 16\, \frac{|g_s|(g-f)(g+f)\sqrt{Q_{\max}}}{g^5}\\
&\leq\frac{-16(1-\ve)^2\,Q_{\max}+C\sqrt{Q_{\max}}}{f^2}
\end{align*}
by using Lemmas~\ref{lem-der} and \ref{lem-sharper-est}.
Because the numerator is negative if $Q_{\max}>\frac{C^2}{256(1-\ve)^2}$, we conclude that
$Q_{\max}\leq C'$. Estimate~\eqref{f-gderiv} follows.
\end{proof}

The results obtained thus far imply that the curvatures of vertical planes
$\kappa_{12}=\kappa_{31}$ and $\kappa_{23}$ become close near a singularity
at a rate that breaks scale invariance.

\begin{corollary}
\label{Cor}
If a solution $\big(\mc S^1\times\mc S^3,G(t)\big)$ satisfies Assumption~\ref{StrongerAssumption},
then there is a uniform constant $C$ such that for as long as the flow exists, one has
\begin{equation}
\label{kappaestimate}
|\kappa_{12}-\kappa_{23}|\leq\frac{C}{M},
\end{equation}
and hence 
\begin{equation}
\label{kappaestimate2}
(T-t)|\kappa_{12}-\kappa_{23}|\leq C\sqrt{T-t}.
\end{equation}
\end{corollary}
\begin{proof}
From the curvature formulas \eqref{Kurv12} and \eqref{Kurv23}, one readily verifies that 
\begin{equation}
\label{kappacalc}
\kappa_{12}-\kappa_{23}=4\frac{(f+g)(f-g)}{g^4}+\frac{g_s(fg_s-gf_s)}{fg^3}.
\end{equation}
Estimate~\eqref{kappaestimate} then follows from \eqref{kappacalc},
 together with Lemmas~\ref{lem-der}, \ref{lem-sharper-est}, and \ref{lem-sharper-der}.
Finally, applying  Lemmas~\ref{lem-upper} and \ref{lem-lower} to \eqref{kappaestimate},
we obtain  estimate \eqref{kappaestimate2}. 
\end{proof}

The mixed sectional curvatures also become close at a rate that breaks scaling.

\begin{lemma}
\label{lem-sharper-k01}
If a solution $\big(\mc S^1\times\mc S^3,G(t)\big)$ satisfies Assumption~\ref{StrongerAssumption},
then there exists a uniform constant $C$ such that for as long as the flow exists, one has
\[|\kappa_{01} - \kappa_{02}| \le \frac{C}{\dn M}.\]
\end{lemma}

\begin{proof}
We define the quantity $k := \frac{f_s}{f} - \frac{g_s}{g}$, observing that 
it follows from Lemma~\ref{lem-sharper-der} that $k=\sqrt Q$ satisfies 
$|k|\leq C$.
Using equations~\eqref{SRFsystem} and \eqref{f_s-evolution}--\eqref{g_s-evolution}, one readily
calculates that $k$ evolves by
\[
k_t = \Delta k - \left(\frac{g_s^2}{g^2} + \frac{f_s^2}{f^2} + \frac{8f^2}{g^4}\right)k
+ \frac{8 g_s\, (f^2 - g^2)}{g^5}.
\]
If we differentiate both sides of this equation with respect to  $s$, use the commutator \eqref{commute},
and recall formula~\eqref{DefineLaplacian} for the Laplacian, we obtain
\begin{equation}
\label{ksderiv}
\frac{\partial}{\partial t}(k_s)= \Delta(k_s) - A\, k_s - B\, k + D,
\end{equation}
where $A$, $B$, and $D$ are functions of $(s,t)$ defined by
\[
A := \frac{g_s^2}{g^2} + \frac{f_s^2}{f^2} + \frac{8f^2}{g^4}, \quad
B := \left\{\frac{g_s^2}{g^2} + \frac{f_s^2}{f^2} + \frac{8f^2}{g^4}\right\}_s, \quad
D := \left\{\frac{8 g_s\, (f^2-g^2)}{g^5}\right\}_s.
\]
Lemmas \ref{lem-upper}, \ref{lem-lower}, \ref{lem-der}, \ref{BoundMixed},
\ref{lem-sharper-est}, and \ref{lem-sharper-der} imply that $A$ and $|B|+|D|$
may be estimated by
\begin{equation}
\label{eq-beh-ABD}
\frac{c}{T-t} \le A \le \frac{C}{T-t} \qquad\mbox{and}\qquad|B| + |D| \le \frac{C}{(T-t)^{3/2}}.
\end{equation}

Using \eqref{ksderiv}, we readily calculate
\[
(k_s^2)_t = \Delta(k_s^2) - 2k_{ss}^2 - 2A\, k_s^2 - 2B\, k k_s + 2D\, k_s.
\]  
Then using the maximum principle, weighted Cauchy--Schwarz, and  estimate~\eqref{eq-beh-ABD},
and recalling that $|k|$ is uniformly bounded, we find that
\begin{align*}
\frac{d}{dt} (k_s^2)_{\max} &\le -2A k_s^2 - 2Bk k_s + 2D k_s \\
&\le -2A k_s^2 + \left(A k_s^2 + \frac{B^2}{A} k^2\right) + \left( Ak_s^2 + \frac{D^2}{A} \right) \\
&\le \frac{C}{(T-t)^2}.
\end{align*}
Integrating this in time, using Lemma~\ref{lem-upper}, and enlarging $C$ if necessary, we get
\begin{equation}
\label{kmax} 
\left|(k_s)_{\max}\right|  \le \frac{C}{\sqrt{T-t}} \le \frac{C}{\dn M}.
\end{equation}
Recalling the definition of $k$, we see that \eqref{kmax} implies that
\[
|k_s| = \left| \frac{f_{ss}}{f} - \frac{g_{ss}}{g} + \frac{g_s^2}{g^2} - \frac{f_s^2}{f^2}\right| \le \frac{C}{\dn M}.
\]
This estimate, together with Lemmas~\ref{lem-der} and \ref{lem-sharper-der}, implies that
\[
|\kappa_{01} - \kappa_{02}| \le \frac{C}{\dn M} + |k|\, \left| \frac{f_s}{f} + \frac{g_s}{g}\right| \le \frac{C}{\dn M},
\]
as desired.
\end{proof}

\section{Local convergence to the shrinking cylinder soliton}	\label{cylinder}
In this section, we demonstrate that solutions originating from initial data that satisfy
Assumption~\ref{StrongerAssumption} converge locally, after parabolic rescaling,
to the rotation- and translation-invariant shrinking cylinder soliton.

\medskip
We begin by deriving an improved $C^1$ bound for the metric component $f$.

\begin{lemma}	\label{lem-fs-1}
If a solution $\big(\mc S^1\times\mc S^3,G(t)\big)$ satisfies Assumption \ref{StrongerAssumption},
then there exists a uniform constant $C$ such that for as long as the flow exists, one has
\[
f_s^2 \le 1 + C\sqrt{T-t}.
\] 
\end{lemma}

\begin{proof}
Based on the evolution equation \eqref{f_s-evolution} for $f_s$, one easily determines that the evolution equation for the quantity $v:=f_s^2$ is given by
\[
v_t = \Delta v - 2f_{ss}^2 - 2 \frac{f_s}{f} v_s
- 2\left(\frac{6f^2}{g^4} + \frac{2g_s^2}{g^2}\right) v + 16 \frac{f^3}{g^5} f_s g_s.
\]
Using Lemmas~\ref{lem-der}, \ref{lem-sharper-est}, and \ref{lem-sharper-der}, we obtain a
uniform constant $C$ such that
\begin{align*}	\label{lem-v-ev-eq}
v_t &\le \Delta v - \frac{2 f_s}{f} v_s - \frac{12 v}{f^2(1+Cf)^4}
- \frac{4 v\big\{f_s + (g_s - f_s)\big\}^2}{f^2 (1+ Cf)^2}
+ \frac{16 f_s \big\{f_s + (g_s - f_s)\big\}}{f^2} \\
&\le \Delta v - \frac{2 f_s}{f} v_s + \frac{4v(1-v)}{f^2} + \frac{C}{f}.
\end{align*}
To get the last inequality, we used the fact that for any $m\geq1$,
\[
\left|\frac{1}{f^2}-\frac{1}{f^2(1+Cf)^m}\right|\leq\frac{C}{f}.
\]

Now we consider the quantity $v_{\max}(t) := \max_{s\in\mc S^1}v(s,t)$, which satisfies
the differential inequality
\[
\frac{d}{dt}v_{\max}\leq\frac{4v(1-v)}{f^2}+\frac{C}{f}.
\]
If $0\leq v_{\max}\leq1$,
there is nothing to prove. If $v_{\max}>1$, then one has
\[
\frac{d}{dt}v_{\max}\leq4\frac{1-v_{\max}}{f^2}+\frac{C}{f},
\]
which implies that $v_{\max}$ is strictly decreasing unless $4(1-v_{\max})+Cf\geq0$,
hence by Lemma~\ref{lem-upper}, strictly decreasing unless
\[
v_{\max}\leq1+C\sqrt{T-t},
\]
where $C$ is a uniform constant.
Combining this inequality with Assumption~\ref{StrongerAssumption}, we conclude that
\[
v_{\max}(t) \le \max\big\{ v_{\max}(0),\, 1 + C\sqrt{T-t}\big\} = 1 + C\sqrt{T-t}.
\]
\end{proof}

As a tool for controlling the second derivative of $f$, we next consider the quantity\,\footnote{This
quantity may be compared to $F$ defined in (25) of \cite{AK04}. In that paper, one has $f=g=\psi$.
So the quantity $F$ in \cite{AK04} simplifies to $F=2\psi\psi_{ss}|\log\psi\,|$ at a neck, and is bounded from above.}
\begin{equation}
\label{Fdef}
F := f f_{ss} \log f,
\end{equation}
and show that it is bounded from below in certain  space-time neighborhoods of a local singularity.
We define the neighborhoods of interest as follows.
For fixed $0<\delta\ll1$, there exists by Lemma~\ref{lem-upper} a time $t_\delta\in[0,T)$
such that the radius of each neck that becomes singular satisfies $f\leq\delta$ for all $t_0\leq t<T$.
Because $f_{ss}>0$ at each local minimum of $f$, the set
\[
\Omega=\left\{f_{ss}\log\left(\frac{f}{\delta}\right)<0\right\}
\]
describes an open interval around that neck (or those necks) for all $t\in(t_\delta,T)$.\footnote{If
there are several equally small necks, $\Omega$ may have several connected components
in space. This does not pose a problem for the argument that follows.}

\begin{lemma}	\label{lem-cyl-est}
If a solution $\big(\mc S^1\times\mc S^3,G(t)\big)$ satisfies Assumption \ref{StrongerAssumption},
then there exists a constant $C$  such that for as long as the flow exists, one has
\[
F \ge -C
\]
in the neck-like region $\Omega$.
\end{lemma}

\begin{proof}
It follows from Lemmas~\ref{lem-lower} and \ref{BoundMixed} that
$F\geq C\log(T-t_\delta)/\sqrt{T-t_\delta}$ at $t=t_\delta$.
Moreover, the definition of $\Omega$ guarantees that  $F=0$ at the endpoints of each component
of $\Omega$ for all times $t_\delta<t<T$. Hence $F$ is uniformly bounded on the parabolic
boundary of $\Omega$. To complete the proof, we show that $F$ is bounded from below
at all interior points. To do so, in the following argument, we use the facts that $f_{ss}>0 $
and $f<\delta$ inside $\Omega$.

Differentiating~\eqref{SRFsystem} using the commutator \eqref{commute},
we compute that $F$ evolves by
\begin{equation}
\label{eq-F}
F_t = \Delta F-2\left(2+\frac{1}{\log f}\right)\frac{f_s}{f}\,F_s + N,
\end{equation}
where the reaction term $N$ is given by 
\begin{align*}
N &:= -f \log f\left(\frac{12 f f_s^2}{g^4} - \frac{48 f^2 f_s g_s}{g^5}
	+ \frac{40 f^3 g_s^2}{g^6} - \frac{4 f_s g_s^3}{g^3}\right)\\
&\quad-\frac{8f^4 \log f}{g^4}\left(\frac{f_{ss}}{f} - \frac{g_{ss}}{g}\right)
- 2 f_{ss}\log f\left( \frac{f^3}{g^4 \log f} + f_{ss}\right)\\
&\quad+ \frac{2 f_s^2 f_{ss}}{f} \left( 2 + \frac{1}{\log f}\right)
- 4f\log f \left(\frac{g_s^2 f_{ss}}{g^2} + \frac{f_s g_s g_{ss}}{g^2} - \frac{f_s^2 f_{ss}}{f^2} \right) .
\end{align*}
To proceed, we estimate the various terms in $N$ one-by-one:

Beginning with the coefficient of $-f \log f$, we observe that there exists a uniform constant $C$ such that
\begin{align*}
\frac{12 f f_s^2}{g^4} - \frac{48 f^2 f_s g_s}{g^5}  + \frac{40 f^3 g_s^2}{g^6} - \frac{4 f_s g_s^3}{g^3}
&\ge \frac{12 f_s^2}{f^3 (1+Cf)^4} - \frac{48 f_s\, (f_s + Cf)}{f^3}\\
&\quad+ \frac{40\big\{f_s + (g_s - f_s)\big\}^2}{f^3(1 + Cf)^6}
- \frac{4 f_s\, \big\{f_s + (g_s - f_s)\big\}^3}{f^3} \\
&\ge \frac{4 f_s^2(1 - f_s^2)}{f^3} - \frac{C}{f^2}\\
&\ge -\frac{C}{f^2}.
\end{align*}
To obtain this estimate, we use Lemmas~\ref{lem-der}, \ref{lem-sharper-est},
and \ref{lem-sharper-der}, and then, in the last step, Lemmas~\ref{lem-lower} and  \ref{lem-fs-1}.
It follows from this estimate, using  Lemma~\ref{lem-lower} again, that\,\footnote{Here we use
the fact that $|\log x\,|/x$ is monotone decreasing for $0<x<1$.}
\begin{equation}	\label{eq-est1}
-f \log f \left(\frac{12 f f_s^2}{g^4} - \frac{48 f^2 f_s g_s}{g^5} 
+ \frac{40 f^3 g_s^2}{g^6} - \frac{4 f_s g_s^3}{g^3}\right) \geq C\,\frac{\log (T-t)}{\sqrt{T-t}}.
\end{equation}
Similarly, relying on Lemma~\ref{lem-sharper-k01}, we obtain
\begin{equation}	\label{eq-est2}
-\frac{8f^4 \log f}{g^4}\left(\frac{f_{ss}}{f} - \frac{g_{ss}}{g}\right)
\geq C\,\frac{\log(T-t)}{\sqrt{T-t}}.
\end{equation}
We deal with $- 2 f_{ss}\log f\left( \frac{f^3}{g^4 \log f} + f_{ss}\right)$ below.
Examining  the fourth term, we observe that at any interior point of $\Omega$, one has
\begin{equation}
\label{eq-est3}
\frac{2 f_s^2 f_{ss}}{f} \left(2 + \frac{1}{\log f}\right) > 0.
\end{equation}
Finally, using the positivity of $f_{ss}$ in $\Omega$ and applying Lemmas~\ref{lem-lower},
\ref{lem-der},  \ref{lem-sharper-est}, \ref{lem-sharper-der}, and \ref{lem-sharper-k01},
we observe that
\begin{align*}
\frac{g_s^2 f_{ss}}{g^2} + \frac{f_s g_s g_{ss}}{g^2} - \frac{f_s^2 f_{ss}}{f^2}
&\geq \frac{f_s g_s g_{ss}}{g^2} - \frac{f_s^2 f_{ss}}{f^2}\\
&=\frac{f_s g_s}{g}\left(\frac{g_{ss}}{g}-\frac{f_{ss}}{f}\right)
+\frac{f_s\big\{f(g_s-f_s)+f_s(f-g)\big\}}{fg}\,\frac{f_{ss}}{f}\\
&\geq-\frac{C}{f\sqrt{T-t}}-C\,\frac{f_{ss}}{f}.
\end{align*}
Combining this estimate with Lemma \ref{lem-lower}, we obtain
\begin{equation}
\label{eq-est4}
-4f\log f\left(\frac{g_s^2 f_{ss}}{g^2} + \frac{f_s g_s g_{ss}}{g^2} - \frac{f_s^2 f_{ss}}{f^2}\right)
\geq C\,\frac{\log(T-t)+F}{\sqrt{T-t}}.
\end{equation}

To proceed, we  assume that $F_{\min}(t)$ is attained at an interior point of $\Omega$.
Using inequalities~\eqref{eq-est1}--\eqref{eq-est4} to estimate the right-hand side of
equation~\eqref{eq-F}, one obtains
\[
\frac{d}{dt}F_{\min} \geq C\,\frac{\log(T-t)+F}{\sqrt{T-t}} - 2 f_{ss}\log f\left(\frac{f^3}{g^4 \log f} + f_{ss}\right).
\]
If $\frac{f^3}{g^4\log f} + f_{ss} \le 0$, then $ff_{ss}|\log f|\leq\frac{f^4}{g^4}$, and so  $F\geq-1$.
Otherwise, we have
\begin{equation*}
\label{Fminderiv}
\frac{d}{dt}F_{\min} \geq C\, \frac{F_{\min}+\log(T-t)}{\sqrt{T-t}}.
\end{equation*}
This inequality, together with the maximum principle, implies that
\begin{align*}
F_{\min}(t)&\geq e^{-C'\sqrt{T-t}}\left\{e^{C'\sqrt{T-t_\delta}}F_{\min}(t_\delta)
+C\int_{t_\delta}^t\frac{\log(T-\tau)\,e^{C'\sqrt{T-\tau}}}
{\sqrt{T-\tau}}\mathrm d\tau\right\}\\
&\geq e^{C'(\sqrt{T-t_\delta}-\sqrt{T-t})}F_{\min}(t_\delta)-C''.
\end{align*}
Since, as noted above, $F_{\min}(t_\delta)\geq C\log(T-t_\delta)/\sqrt{T-t_\delta}$, the proof is complete.
\end{proof}

Using  Lemmas~\ref{lem-upper} and \ref{lem-lower}, we obtain the following
consequence of Lemma~\ref{lem-cyl-est}.

\begin{corollary}	\label{SmallCurvature}
In the neighborhood $\Omega$ of the smallest neck(s), where the sectional curvature
$\kappa_{01}$ is negative, the scale-invariant quantities $(T-t)|\kappa_{01}|$ and $(T-t)|\kappa_{02}|$ satisfy
\[
(T-t)|\kappa_{01}|\leq \frac{C}{|\log(T-t)|} \qquad\mbox{and}\qquad (T-t)|\kappa_{02}| \leq \frac{C}{|\log(T-t)|}.
\]
\end{corollary}

\begin{proof}
By Lemmas \ref{lem-upper} and \ref{lem-lower}, the estimate for $\kappa_{01}$ is a straightforward
consequence of Lemma \ref{lem-cyl-est}. Then combining Lemma~\ref{lem-sharper-k01} with this
estimate for  $\kappa_{01}$,  we obtain 
\begin{align*}
(T-t) |\kappa_{02}| &\le (T-t) |\kappa_{01}| + (T-t) |\kappa_{01} - \kappa_{02}|\\
&\le \frac{C}{|\log(T-t)|} + C\sqrt{T-t}\\
&\le \frac{C}{|\log(T-t)|}.
\end{align*}
\end{proof}

We now prove cylindricality at a singularity. Without loss of generality, we confine our considerations
to a single component of $\Omega$. We choose $\xi_1(t)$ such that $f(s(\xi_1(t),t),t)=\dn M(t)$ in that
component for all times $t$ sufficiently close to $T$, and we define arclength from the neck by
\[
S(\xi,t):=s(\xi,t)-s(\xi_1(t),t).
\]

\begin{lemma}	\label{Cylinder}
There exist uniform constants $0<\ve<1$ and $c,\,C<\infty$ such that for all times $t$
sufficiently close to $T$, one has
\[
1\leq\frac{f}{\dn M}\leq1+C\,\frac{(S/\dn M)^2}{\big|\log\dn M\big|}
\qquad\mbox{and}\qquad
1\leq\frac{g}{\dn M}\leq \big(1+o(1)\big)\left(1+C\,\frac{(S/\dn M)^2}{\big|\log\dn M\big|}\right)
\]
for $|S|\leq c\dn M\sqrt{|\log\dn M|}$, and
\[
\frac{f}{\dn M}+\frac{g}{\dn M}\leq C\,\frac{|S/\dn M|}{\sqrt{|\log\dn M|}}
\sqrt{\log\left(\frac{|S/\dn M|}{\sqrt{|\log \dn M|}}\right)}
\]
for $c\dn M\sqrt{|\log\dn M|}\leq|S|\leq\dn M^{1-\ve}$.
\end{lemma}

\begin{proof}
We carry out  the argument for the side of the neck on which $S\geq0$; the other side is treated analogously.
Because $f_s>0$ where $S>0$, we can use $f$ as a coordinate there. More precisely, we use $\ell:=\log f$.
Then because $\frac{ds}{d\ell}=f/f_s$, we can state the conclusion of Lemma~\ref{lem-cyl-est} as
\[
\frac{\partial}{\partial\ell}(f_s^2)=2ff_{ss}\leq-\frac{C}{\ell}.
\]
Integrating this inequality and using the calculus fact that $\log x\leq x-1$ for $x\geq1$, we obtain
\[
f_s^2\leq C\log\left(\frac{\log\dn M}{\log f}\right)\leq C\left(\frac{\log\dn M}{\log f}-1\right).
\]
This estimate implies that
\[
\sqrt{C}\,\frac{dS}{d\ell}\geq\frac{f}{\sqrt{\frac{\log\dn M}{\log f}-1}}
=\frac{\frac{df}{d\ell}}{\sqrt{\frac{\log\dn M}{\ell}-1}},
\]
which upon another integration yields
\[
\sqrt{C}\,S\geq\int_{\dn M}^f \frac{d\tilde f}{\sqrt{\frac{\log\dn M}{\log \tilde f}-1}}.
\]
We change the variable of integration to $\vp=\tilde f/\dn M$, obtaining
\[
\sqrt{C}\,S\geq\dn M\int_{1}^{f/\dn M}\sqrt{\frac{-\log\dn M-\log\vp}{\log\vp}}\;\mathrm d\vp.
\]
Restricting to a smaller neighborhood of the neck if necessary so that $f\leq\dn M^{-3/4}$,
we ensure that $\sqrt{-\log\dn M-\log\vp}\geq\frac12\sqrt{-\log\dn M}$ and so obtain the
simpler estimate
\[
2\sqrt{C}\,\frac{S}{\dn M\sqrt{-\log\dn M}}\geq\int_{1}^{f/\dn M}
\frac{\mathrm d\vp}{\sqrt{\log\vp}}.
\]
As observed in Proposition~9.3 of \cite{AK04}, this inequality implies that
\[\frac{f}{\dn M}\leq1+C'\,\frac{(S/\dn M)^2}{\big|\log\dn M\big|}
\]
for $S\leq c\dn M\sqrt{|\log\dn M|}$, and
\begin{equation}	\label{Big-f}
\frac{f}{\dn M}\leq C''\,\frac{S/\dn M}{\sqrt{|\log\dn M|}}
\sqrt{\log\left(\frac{S/\dn M}{\sqrt{|\log \dn M|}}\right)}
\end{equation}
for larger values of $S$.

To obtain the estimates for $g$, we argue as follows. Because $f$ is monotone increasing
moving away from the neck in $\Omega$, we may use estimate~\eqref{Big-f} to see that if
$S\leq\dn M^{1-\ve}$ for $\ve\in(0,1)$, then $f=o(1)$ as $\dn M\searrow0$. Hence by
Lemma~\ref{lem-sharper-est}, we obtain $g\leq(1+Cf)f\leq\big(1+o(1)\big)f$ as
$\dn M\searrow0$.
\end{proof}

Lemma~\ref{lem-lower}, Corollary~\ref{SmallCurvature}, and Lemma~\ref{Cylinder} imply that a
Type-I blowup of the metric,
\[
\tilde G:=(T-t)^{-1}G,
\]
must converge near the singularity to the shrinking cylinder soliton. It follows that
\[
\dn M=\big(1+o(1)\big)\,2\sqrt{T-t}.
\]
If we now denote  the parabolically-rescaled distance from the neck by
\[
\sigma:=\frac{S}{\sqrt{T-t}},
\]
then  the conclusion of Lemma~\ref{Cylinder} may be  recast as follows.

\begin{corollary}	\label{NicestYet}
There exist uniform constants $0<\ve<1$ and $c,\,C<\infty$ such that as $t\nearrow T$,
the estimates
\[
1+o(1)\leq\frac{f}{2\sqrt{T-t}}\leq1+C\,\frac{\sigma^2}{|\log(T-t)|}
\]
and
\[
1+o(1)\leq\frac{g}{2\sqrt{T-t}}\leq\big(1+o(1)\big)
\left(1+C\,\frac{\sigma^2}{|\log(T-t)|}\right)
\]
hold for $|\sigma|\leq c\sqrt{|\log(T-t)|}$, and the estimate
\[
\frac{f}{\sqrt{T-t}}+\frac{g}{\sqrt{T-t}}\leq C\,\frac{|\sigma|}{\sqrt{|\log(T-t)|}}
\sqrt{\log\left(\frac{|\sigma|}{\sqrt{|\log (T-t)|}}\right)}
\]
holds for $c\sqrt{|\log(T-t)|}\leq|\sigma|\leq(T-t)^{-\ve/2}$.
\end{corollary}

\section{Estimates for reflection-symmetric solutions}
\label{reflection}

In this section, we derive our sharpest estimates for the eccentricity of a Ricci flow solution
near a developing neckpinch, more than doubling the decay rate for the scale-invariant
quantity $|f-g|/\sqrt{T-t}$ that we have obtained above. To accomplish this, we use ideas motivated
by the formal asymptotics outlined in Appendix~\ref{FormalNeckpinch}, following
the approach carried out rigorously in \cite{AK07}. To make the arguments rigorous
here, we impose Assumption~\ref{Reflection} from Section \ref{Intro},
adding a technical hypothesis that guarantees that each solution under
consideration is reflection symmetric, with its smallest neck occurring at $s=0$.
In this approach, we find that the evolution of the quantity we study below, which controls
$|f-g|$, is governed by a favorable linear term and by a ``forcing function'' that represents the
nonlinear terms involved. As in \cite{AK07}, we do not quite achieve the optimal decay predicted
by the linear term, but we are able to prove decay at the rate of the forcing function.

Our first step, which does not need reflection symmetry, is a mild improvement to 
Lemma~\ref{lem-sharper-est}, to be used below.

\begin{lemma}	\label{LogFactor}
If a solution $\big(\mc S^1\times\mc S^3,G(t)\big)$ satisfies Assumption \ref{StrongerAssumption},
then there exists a uniform constant $C$ such that for as long as the flow exists, one has
\[
g-f\leq Cf^3|\log(T-t)|.
\]
\end{lemma}

\begin{proof}
We define $P:=f^{-2}-g^{-2}>0$ and compute that
\[
P_t=\Delta P+g^2\left(\frac{4f_s^2}{f^3}+P_s\right)P_s
+\frac{4(g^2-f^2)(g^6f_s^2-f^6)}{f^6g^6}.
\]
Therefore, using the fact that $f\leq g$, one has
\[
\frac{d}{dt}P_{\max}\leq\frac{4(g^2-f^2)}{f^6g^6}
\Big\{g^6-f^6+g^6(f_s^2-1)\Big\}.
\]
Now by Lemma~ \ref{lem-sharper-est}, one has
$g^6-f^6=(g^3+f^3)(g^2+fg+f^2)(g-f)\leq Cf^7$.
Then using Lemma~\ref{lem-fs-1}, Lemma~\ref{lem-sharper-est} again, and finally
Lemma~\ref{lem-lower}, one obtains
\begin{align*}
\frac{d}{dt}P_{\max}
&\leq C\frac{g^2-f^2}{f^6g^6}\big\{f^7+g^6\sqrt{T-t}\big\}\\
&\leq C\left\{\frac{1}{f^2}+\frac{\sqrt{T-t}}{f^3}\right\}\\
&\leq \frac{C}{T-t}.
\end{align*}
Integrating this yields $P_{\max}\leq C\big\{1-\log(T-t)\big\}\leq C'|\log(T-t)|$, whereupon
unwrapping the definition of $P$ and using Lemma~\ref{C0pinching} gives the result in the form
\[
g-f=\frac{f^2g^2}{f+g}P\leq Cf^3P_{\max}\leq C'f^3|\log(T-t)|.
\]
\end{proof}

Next we perform a parabolic dilation as outlined in Appendix~\ref{FormalNeckpinch},
the purpose of which is to facilitate analysis of the solution very near the developing
singularity, following the approach of \cite{AK07}. We introduce new time $\tau:=-\log(T-t)$
and space  $\sigma:=e^{\tau/2}s$ variables. Then we consider the quantity $x(\sigma,\tau)$
defined in equation~\eqref{Define-xy} (found in Appendix \ref{FormalNeckpinch}), which is
\[
x=\frac12e^{\tau/2}(f-g)=\frac{f-g}{2\sqrt{T-t}}.
\]
As computed in Appendix~\ref{FormalNeckpinch}, the evolution of $x$ is governed  by
\[
x_\tau=(\mc A-3)x+N(x),
\]
where the familiar linear operator
\[
\mc A:=\frac{\partial^2}{\partial\sigma^2}-\frac{\sigma}{2}\frac{\partial}{\partial\sigma}+1
\]
generates the quantum harmonic oscillator. The nonlinear quantity
$N(x)$ is
\begin{multline}		\label{Define-N}
N(x) := \frac{\vp_\sigma \psi_\sigma}{(1+\vp) (1 + \psi)}x + \mc I x_{\sigma}\\
	- \frac{\vp^2 + 2\vp(2 + \psi) - \psi\big\{14 + \psi[28 + 5\psi(4 +\psi)]\big\}}{2(1 + \psi)^4}x,
\end{multline}
where
\[
\vp:=u-1:=\frac{e^{\tau/2}f}{2}-1,\qquad\qquad\psi:=v-1:=\frac{e^{\tau/2}g}{2}-1,
\]
and $\mc I(\sigma,\tau)$ is the nonlocal term
\begin{equation}	\label{Define-I}
\mc I:= \int_0^\sigma \left(\frac{u_{\bar\sigma \bar\sigma}}{u}
	+2\frac{v_{\bar\sigma \bar\sigma}}{v}\right)\mr d\bar\sigma.
\end{equation}

In order to estimate the nonlinear terms above, we need the following analog of
Lemma~4 from \cite{AK07}. Note that this is the first time we use our strongest
assumption on the initial data, that of reflection symmetry.

\begin{lemma}	\label{C0C1bounds}
If a solution $\big(\mc S^1\times\mc S^3,G(t)\big)$ satisfies Assumption \ref{Reflection},
then there exist $\ve\in(0,1)$ and $c,C$ such that one has $C^0$ estimates
\begin{align*}
1 - \frac{C}{\tau} &\leq u \leq 1 + C\frac{\sigma^2}{\tau}, \qquad\qquad\; |\sigma| \le c\sqrt{\tau},\\
1 - \frac{C}{\tau} &\leq u \leq C\frac{|\sigma|}{\sqrt{\tau}} \sqrt{\log\frac{|\sigma|}{\sqrt{\tau}}},
	\qquad c\sqrt{\tau} \leq |\sigma| \le e^{\ve\tau},\\ \\
1 - \frac{C}{\tau} &\leq v \leq 1 + C\frac{1+\sigma^2}{\tau}, \qquad\quad\; |\sigma| \le c\sqrt{\tau},\\
1 - \frac{C}{\tau} &\leq v \leq C\frac{|\sigma|}{\sqrt{\tau}} \sqrt{\log\frac{|\sigma|}{\sqrt{\tau}}},
	\qquad c\sqrt{\tau} \leq |\sigma| \le e^{\ve\tau},
\end{align*}
and $C^1$ estimates
\begin{align*}
|u_{\sigma}| + |v_{\sigma}|&\leq C\frac{1+|\sigma|}{\tau},\qquad\qquad\quad|\sigma| \le c\sqrt{\tau},\\
|u_{\sigma}| + |v_{\sigma}|&\leq \frac{C}{\sqrt{\tau}}\, \sqrt{\log \frac{|\sigma|}{\sqrt{\tau}}},
	\qquad\quad c\sqrt{\tau|} \le |\sigma| \le e^{\ve\tau}.
\end{align*}
\end{lemma}

\begin{proof}
The upper bound for $u$ follows immediately from Corollary~\ref{NicestYet}.

To get the lower bound for $u$, we note that at the center of the neck, Lemma~\ref{lem-cyl-est}
and Lemma~\ref{lem-lower} imply that
\[
ff_{ss}\leq \frac{C}{|\log f|}\leq\frac{C'}{\tau}.
\]
Then using the implication of Lemma~\ref{lem-upper} that
\[
\frac{f^4}{g^4}=1+\frac{f^4-g^4}{g^4}
=1+\frac{(f-g)(f^3+f^2g+fg^2+g^3)}{g^4}\geq1-Cf\geq1-C'e^{-\tau/2},
\]
we observe that at the center of the neck, where $f$ achieves its minimum, one has
\[
\frac12(f^2)_t=ff_t=ff_{ss}-2\frac{f^4}{g^4}\leq-2+\frac{C}{\tau},
\]
which implies after integration that $f^2\geq4(T-t)\big(1-C\tau^{-1}\big)$, hence that
\[
u\geq\sqrt{1-\frac{C}{\tau}}\geq1-\frac{C'}{\tau}.
\]

The upper bound for $v$ at large $|\sigma|$ is implied by Corollary~\ref{NicestYet}.
To get the upper bound at small $|\sigma|$, we note that by Lemma~\ref{lem-sharper-est},
one has $v\leq(1+Cf)u$. But for $|\sigma|\leq c\sqrt\tau$, the upper bound for $u$ implies
that $f\leq Ce^{-\tau/2}$, which in turn implies the estimate.

To get the lower bound for $v$, we note that Lemma~\ref{lem-sharper-k01} implies that
\[
g_{ss}\leq\frac{g}{f}\big(f_{ss}+C\big).
\]
Therefore at the center of the neck, where $g$ also achieves its minimum, we apply Lemmas~\ref{lem-upper},
\ref{lem-sharper-est}, and \ref{lem-cyl-est} to obtain
\begin{align*}
\frac12(g^2)_t&=gg_{ss}+2\frac{f^2-g^2}{g^2}-2\\
&\leq\frac{g^2}{f^2}(ff_{ss})+Cf-2\\
&\leq-2+\frac{C}{\tau}.
\end{align*}
In the final step of the derivation of this estimate, we have used the fact that $f=\dn M\leq Ce^{-\tau/2}$ at the center of the neck. Working with this estimate, we derive the lower bound for $v$ using an argument very similar to that used to obtain the lower bound for $u$.

The derivative bounds in the statement of the Lemma follow readily from the $C^0$ and $C^2$ bounds
we have obtained above, together with  the fact that the definition of $\Omega$ ensures that $u$ and $v$ are
 convex there. For example, unwrapping definitions and using the estimate $ff_{ss}|\log f\,|\leq C$
from Lemma~\ref{lem-cyl-est}, one sees that
\begin{equation}
\label{uu}
uu_{\sigma\sigma}|\log(2e^{-\tau/2}u)|=\frac14 ff_{ss}|\log f|\leq C.
\end{equation}
Then for $|\sigma|\leq c\sqrt\tau$, a region in which Corollary~\ref{NicestYet} implies that
$u=\mc O(1)$, our estimate~\eqref{uu} implies that
$u_{\sigma\sigma}\leq C/\tau$, which after antidifferentiation yields
\[
|u_\sigma|\leq C\frac{|\sigma|}{\tau},\qquad\qquad\big(|\sigma|\leq c\sqrt\tau\big).
\]
The remaining bounds are proved similarly.
\end{proof}

The operator $-\mc A$ is self-adjoint in the Hilbert space
$\mc G:=L^2(\mathbb R;\,e^{-\sigma^2/4}\mathrm d\sigma)$, with discrete spectrum
bounded below by $-1$. We denote the inner product in $\mc G$ by
$\lh\cdot,\cdot\rh$ and the norm by $\|\cdot\rn$.

The quantity $x$ does not belong to $\mc G$ because it
is not defined for all $\sigma$. We remedy this difficulty as follows. Let $\beta$ be a smooth, even,
bump function with $\beta(z)=1$ for $|z|\leq1$ and $\beta(z)=0$ for $|z|\geq2$. We define
\[
X(\sigma,\tau):=\left\{\begin{array}[c]{cc}

\beta(e^{-\ve\tau/2}\sigma)\,x(\sigma,\tau), & \text{for }|\sigma|\leq2e^{\ve\tau/2},\\
0 & \text{for }|\sigma|>2e^{\ve\tau/2},\end{array}\right.
\]
where $\ve$ is the constant from Corollary~\ref{NicestYet}. A computation shows
that
\[
X_\tau=(\mc A-3)X+\beta N(x)+E,
\]
where $E:=\big(\beta_\tau-\beta_{\sigma\sigma}+\frac{\sigma}{2}\beta_\sigma\big)x
-2\beta_\sigma x_\sigma$ denotes the ``error'' induced by $\beta$.

\begin{lemma}	\label{E-estimate}
The quantity $E$ vanishes except for $e^{\ve\tau/2} < |\sigma| <  2e^{\ve\tau/2}$,
and there exists a uniform constant $C$ such that if $E\neq 0$, then
\[
  |E(\sigma,\tau)|\leq C|\sigma|\qquad\mbox{and}\qquad
  \| E (\cdot,\tau) \rn \leq C \exp\left(-e^{\ve\tau/2}\right).
\]
\end{lemma}

As a consequence of Lemma~\ref{C0C1bounds}, the proof of this result is identical to that of Lemma~7
in \cite{AK07}.
\medskip

We are now ready to estimate the evolution of $\|X\rn^2$. In doing this, we use the fact that for
any function $W$ in the domain of the operator $\mc A$, the divergence form of that operator shows that
$\mc A W = (e^{-\sigma^2/4}W_\sigma)_\sigma\,e^{-\sigma^2/4}$. Consequently one has
\[
-\lh W,\mc A W\rh = \int_{\mathbb R}(W_\sigma^2-W^2)\,e^{-\sigma^2/4}\,\mathrm d\sigma,
\]
and hence
\[
\|W_\sigma\rn^2=\|W\rn^2-\lh W,\mc A W\rh.
\]
Thus we obtain
\begin{align*}
\frac{d}{d\tau}\|X\rn^2
&=2\lh X,X_\tau\rh\\
&=2\lh X,\,\mc AX-3X+\beta N(x)+E\rh\\
&=-4\|X\rn^2-2\|X_\sigma\rn^2+2\lh X,\,\beta N(x)+E\rh.
\end{align*}
We now define $N_0(x):=N(x)-\mc I x_\sigma$, where $N(x)$ is
defined in~\eqref{Define-N} and $\mc I$ is defined in~\eqref{Define-I}.
We also define $E_0:=\beta_\sigma x$. It then follows from Cauchy--Schwarz that we have
\begin{align}
\frac{d}{d\tau}\|X\rn^2
&=-4\|X\rn^2-2\|X_\sigma\rn^2
	+2\lh X,\,\mc I X_\sigma+N_0(x)\rh+2\lh X,\,E-\mc I E_0\rh	\notag\\
&\leq-4\|X\rn^2+\|\mc I X\rn^2
	+2\lh X,\, N_0(x)\rh+2\lh X,\,E-\mc I E_0\rh.	\label{X-evolve}
\end{align}
\medskip

To control $\frac{d}{d\tau}\|X\rn^2$, we start by deriving pointwise bounds
for the nonlinear factors on the right hand side of ~\eqref{X-evolve}.

\begin{lemma}	\label{estimate-N}
If $\big(\mc S^1\times\mc S^3,G(t)\big)$ is a Ricci flow solution satisfying
Assumption~\ref{Reflection}, then for $|\sigma|\leq c\sqrt\tau$, one has
\[
|\mc I|\leq C\left(\frac{1}{\tau}+\frac{|\sigma|^3}{\tau^2}\right)\qquad\mbox{and}\qquad
|N_0|\leq C\left(\frac{1}{\tau}+\frac{\sigma^8}{\tau^4}\right)|x|;
\]
while for $c\sqrt\tau\leq|\sigma|\leq e^{\ve\tau}$, one has
\[
|\mc I|\leq C\frac{|\sigma|}{\tau}\log\frac{|\sigma|}{\sqrt\tau}
\qquad\mbox{and}\qquad
|N_0|\leq C \frac{\sigma^4}{\tau^2}\left(\log\frac{|\sigma|}{\sqrt\tau}\right)^2|x|.
\]
\end{lemma}

\begin{proof}
Our assumption of reflection symmetry allows us to integrate by parts and
thus write the nonlocal term $\mc I$ defined in ~\eqref{Define-I} as
\begin{align*}
\mc I
&=\frac{u_\sigma}{u}+2\frac{v_\sigma}{v}\;
	+\int_0^\sigma\frac{u_{\bar\sigma}^2}{u^2}\,\mathrm d\bar\sigma
	+2\int_0^\sigma\frac{v_{\bar\sigma}^2}{v^2}\,\mathrm d\bar\sigma.
\end{align*}
We now use $\vp=u-1$ and $\vp_\sigma=u_\sigma$, together with $\psi=v-1$ and $\psi_\sigma=v_\sigma$, 
and proceed to estimate the terms above using Lemma~~\ref{C0C1bounds}. This yields the
stated bounds for $|\mc I|$.

The bounds for  $|N_0|$ also follow easily from Lemma~~\ref{C0C1bounds}.
For $|\sigma|\leq c\sqrt\tau$, one has
\[
|N_0|\leq C\left\{\frac{1+\sigma^2}{\tau^2}+\frac{1+\sigma^2}{\tau}
	+\left(\frac{1+\sigma^2}{\tau}\right)^4\right\}|x|
	\leq C'\left(\frac{1}{\tau}+\frac{\sigma^8}{\tau^4}\right)|x|.
\]
The bound on $|N_0|$ for $c\sqrt\tau\leq|\sigma|\leq e^{\ve\tau}$ is obtained similarly.
\end{proof}

\begin{lemma}
If $\big(\mc S^1\times\mc S^3,G(t)\big)$ is a Ricci flow solution satisfying
Assumption~\ref{Reflection}, then for any $\delta$ sufficiently small, there exist $C$
and $\tau^*$ depending on $\delta$ such that for all times $\tau\geq\tau^*$, one has
\[
\frac{d}{d\tau}\|X\rn^2\leq-(4-\delta)\|X\rn^2+C e^{-2(1+\delta)\tau}.
\]
\end{lemma}

\begin{proof}
We estimate the terms on the \textsc{rhs} of \eqref{X-evolve},
starting with $\lh X,\,E-\mc I E_0\rh$. Given any $\delta_1>0$,
we find by using  weighted Cauchy--Schwarz and Lemma~\ref{E-estimate} that
\begin{align*}
\left|\lh X,E\rh\right|&\leq\frac{\delta_1}{4}\|X\rn^2+\frac{1}{\delta_1}\|E\rn^2\\
&\leq\frac{\delta_1}{4}\|X\rn^2+\frac{C}{\delta_1}\exp(-2e^{\ve\tau/2})\\
&\leq\frac{\delta_1}{4}\|X\rn^2+Ce^{-2\tau}
\end{align*}
for all $\tau\geq\tau_1$, where $\tau_1$ is chosen sufficiently large,
depending only on $\delta_1$ and $\ve$. Then using the facts that
$E_0=\beta_\sigma x$ is supported in $e^{\ve\tau/2} < |\sigma| <  2e^{\ve\tau/2}$,
that $|\beta_\sigma|$ is bounded, and that Lemma~\ref{C0C1bounds}
provides bounds for $|x|$ in that region, we apply Lemma~\ref{estimate-N} and thereby obtain
the estimate 
\[
\|\mc I E_0\rn\leq C\int_{e^{\ve\tau/2}}^\infty\frac{\sigma^2}{\tau}
	\left({\log\frac{|\sigma|}{\sqrt\tau}}\right)e^{-\sigma^2/4}\,\mathrm d\sigma
	\leq C'\exp(-e^{\ve\tau/2}),
\]
exactly as in the proof of Lemma~\ref{E-estimate}. Consequently, arguing as above,
we obtain
\[
\left|\lh X,\mc I E_0\rh\right|\leq\frac{\delta_1}{4}\|X\rn^2+Ce^{-2\tau}.
\]
\smallskip

Next we decompose $\|\mc I X\rn^2+2\lh X,N_0(x)\rh$ into the sum of two quantities
$\mc J_1$ and $\mc J_2$, defined as
\begin{align*}
\mc J_1&:=
\int_{|\sigma|\leq \ve_1\sqrt\tau}\big\{\mc I X^2+2N_0(x)X\big\}\,e^{-\sigma^2/4}\,\mathrm d\sigma,\\
\mc J_2&:=
\int_{\ve_1\sqrt\tau\leq|\sigma|\leq2e^{\ve\tau/2}}\big\{\mc I X^2+2N_0(x)X\big\}\,e^{-\sigma^2/4}\,\mathrm d\sigma,
\end{align*}
for $\ve_1\leq c$ to be chosen. Since, as a consequence of  Lemma~\ref{estimate-N},
we have $|\mc I|\leq C$ and $|N_0(x)|\leq CX$ for $|\sigma|\leq c$, we can enforce the inequality 
$|\mc J_1|\leq\frac{\delta_1}{4}\|X\rn^2$ by choosing $\ve_1$ sufficiently small,
depending only on $\delta_1$ and $C$. Thus in what follows, we focus on $\mc J_2$.

Using Lemma~\ref{LogFactor} along with the weaker growth estimates in Lemma~\ref{C0C1bounds},
we get
\[
|x|=\frac12 e^{\tau/2}(g-f)
	\leq C\frac{|\log(T-t)|}{\sqrt{T-t}}f^3
	= C\tau e^{-\tau}u^3
	\leq C\tau e^{-\tau}\left(1+\frac{\sigma^6}{\tau^3}\right).
\]
So by Lemma~\ref{estimate-N}, we have
\[
|\mc I|X^2\leq Ce^{-2\tau}\left(\tau+\frac{|\sigma|^{15}}{\tau^6}\right)\qquad
\mbox{and}\qquad
|N_0(x)X|\leq Ce^{-2\tau}\left(\tau+\frac{\sigma^{20}}{\tau^8}\right).
\]
We now fix $0<\delta_2\leq1/2$. Then there exist constants $C$ depending on $\delta_2$ such that
\begin{align*}
|\mc J_2|&\leq Ce^{-2\tau}\int_{\ve_1\sqrt\tau\leq|\sigma|\leq2e^{\ve\tau/2}}
	\left(\tau+\frac{\sigma^{20}}{\tau^8}\right)e^{-\delta_2\sigma^2/4}\,
	e^{-(1-\delta_2)\sigma^2/4}\mathrm d\sigma\\
&\leq Ce^{-2\tau}\int_{\ve_1\sqrt\tau\leq|\sigma|\leq2e^{\ve\tau/2}}
	e^{-(1-\delta_2)\sigma^2/4}\mathrm d\sigma\\
&\leq Ce^{-2\tau}\,e^{-\frac{1-\delta_2}{4}\ve_1^2\tau}.
\end{align*}
Thus if we set $\delta_3:=(1-\delta_2)\ve_1^2/8$, we get
$|\mc J_2|\leq Ce^{-2(1+\delta_3)\tau}$, whence the result follows.
\end{proof}

If $\delta\leq2/3$, then the \textsc{ode} for $\|X\rn^2$ implies that for all
$\tau\geq\tau^*$, one has $\|X\rn\leq Ce^{-(1+\delta)\tau}$. The regularizing
effect of the heat equation lets us bootstrap this estimate by one spatial derivative.
The proof is nearly identical to that of Lemma~9 in \cite{AK07} but is even simpler,
because the operator $\mc A-3$ has no unstable eigenmodes. Consequently, we
omit the details, and state our result as follows:

\begin{corollary}
If $\big(\mc S^1\times\mc S^3,G(t)\big)$ is a Ricci flow solution satisfying
Assumption~\ref{Reflection}, then for any $\delta$ sufficiently small, there exist
$C$ and $\tau^*$ depending on $\delta$ such that for all $\tau\geq\tau^*$, one has
\[
\|X\rn +\|X_\sigma\rn \leq Ce^{-(1+\delta)\tau}.
\]
\end{corollary}

Using  Sobolev embedding, we find  that this result implies that on bounded
$|\sigma|$ intervals, one has a pointwise estimate
$|x|=|X|\leq Ce^{\sigma^2/8}\big(\|X\rn+\|X_\sigma\rn\big)$.
We therefore reach the following conclusion.

\begin{corollary}
\label{Corol}
If $\big(\mc S^1\times\mc S^3,G(t)\big)$ is a Ricci flow solution satisfying
Assumption~\ref{Reflection}, then for any $\delta$ sufficiently small and
$\Sigma$ large, there exist $C$ and $\tau^*$ depending on $\delta$ and $\Sigma$
such that for all $|\sigma|\leq\Sigma$ and $\tau\geq\tau^*$, one has
\[
\frac{g-f}{\sqrt{T-t}}=2|x|\leq C e^{-(1+\delta)\tau}=C(T-t)^{1+\delta}.
\]
\end{corollary}

Thus as promised above, we obtain an improved rate of decay for the scale-invariant quantity
$|f-g|/\sqrt{T-t}$. This decay indicates that the solution is rapidly approaching roundness in a
spatial neighborhood of the center of the neck.

\appendix

\section{Warped Berger metrics}	\label{Ansatz}

\subsection{The metrics we study}

To begin, we identify $\mc S^3$ with the Lie group $\mr{SU}(2)$, and consider
general left-invariant metrics of the form
\[
 \up G = f^2\omega^1\ten\omega^1
 	   + g^2\omega^2\ten\omega^2
	   + h^2\omega^3\ten\omega^3,
\]
where the coframe $(\omega^1,\omega^2,\omega^3)$ is algebraically dual to a fixed Milnor frame
$(F_1,F_2,F_3)$ (see \cite[Chapter~1]{CK04}). The sectional curvatures of $\up G$ are then
\begin{align*}
 \up\kappa_{12} &= \frac{(f^2-g^2)^2}{(fgh)^2}-3\frac{h^2}{(fg)^2}+\frac{2}{f^2}+\frac{2}{g^2},\\
 \up\kappa_{23} &= \frac{(g^2-h^2)^2}{(fgh)^2}-3\frac{f^2}{(gh)^2}+\frac{2}{g^2}+\frac{2}{h^2},\\
 \up\kappa_{31} &= \frac{(f^2-h^2)^2}{(fgh)^2}-3\frac{g^2}{(fh)^2}+\frac{2}{f^2}+\frac{2}{h^2}.
\end{align*}

Notice that setting $f^2=\ve$ and $g^2=h^2=1$ recovers the classic Berger collapsed sphere,
shrinking the fibers of the Hopf vibration $\mc S^1\hookrightarrow\mc S^3\twoheadrightarrow\mc S^2$
with sectional curvatures $\kappa_{12}=\kappa_{31}=\ve$ and $\kappa_{23}=4-3\ve$. 

Now we consider Riemannian manifolds $(\mc M^4,G)$ having metrics of the form
\begin{align*}
 G &= \rho^2(\dx)^2+\up G(\xi)\\
     &= (\ds)^2+f(s)^2\omega^1\ten\omega^1
 	   + g(s)^2\omega^2\ten\omega^2
	   + h(s)^2\omega^3\ten\omega^3,
\end{align*}
where $\ds:=\rho\,\dx$. We assume for now that $f,g,h$ depend only on $s(\xi)$, where
$-\infty\leq\xi_-<\xi<\xi_+\leq\infty$. We suppress dependence on $\xi$ when
possible, regarding $f,g,h$ as functions of $s\in\mc B:=(s_-,s_+)$, where $s_-=s(\xi_-)$
and $s_+=s(\xi_+)$. We leave the boundary conditions at $s_\pm$, hence the topology
of $\mc M^4$, open for now. We call these geometries \emph{warped Berger metrics.}
They are generalized warped products; for example, setting $f=g=h$ recovers the warped
products studied in \cite{AK04}.

Note that $\pi:(\mc M^4,G)\rightarrow(\mc B,\dn G)$ is a Riemannian submersion, where
$\dn G=(\ds)^2$. As in \cite{ON66}, for each $x\in\mc M^4$
with $s=\pi(x)$, we define fibers $\mc F_x=\pi^{-1}(s)=(\mc S^3,\up G(s))$. We identify
$\mc V_x=\ker\pi_*:T_x\mc M^4\rightarrow T_s\mc B$ with $T_x\mc F_x\subset T_x\mc M^4$.
We set $\mc H_x=\mc V_x^\perp\subset T_x\mc M^4$. One calls $\mc V$ and $\mc H$ the
vertical and horizontal distributions, respectively.

We define $F_0:=\frac{\partial}{\partial s}$, and let $F_1,F_2,F_3$ be the Milnor frame introduced
above. Hereafter, we let Greek indices range in $0,\dots,3$ and Roman indices in $1,\dots,3$. We
denote the components of the curvature tensor $\Rm$ of $G$ by $R_{\alpha\beta\lambda\mu}$,
and those of the curvature tensor $\up{\Rm}$ of $\up G$ by $\up R_{ijk\ell}$. Then
$\up R_{1221}=(fg)^2\up\kappa_{12}$, $\up R_{2332}=(gh)^2\up\kappa_{23}$, and
$\up R_{3113}=(fh)^2\up\kappa_{31}$.

\subsection{Curvatures of the total space}	\label{ComputeCurvatures}
Recall that O'Neill \cite{ON66} introduces $(2,1)$ tensor fields $A$ and $T$ that act on vector
fields $M,N$ by
\begin{equation}
 \label{define-A}
 A_M N = \mc H\big\{\cv_{(\mc H M)}(\mc V N)\big\}
 	+\mc V\big\{\cv_{(\mc H M)}(\mc H N)\big\},
\end{equation}
and
\begin{equation}
 \label{define-T}
 T_M N = \mc H\big\{\cv_{(\mc V M)}(\mc V N)\big\}
 	+\mc V\big\{\cv_{(\mc VM)}(\mc H N)\big\},
\end{equation}
respectively. Using these, one computes the tensor $\Rm$ from $\up{\Rm}$, as follows.

We denote the connection $1$-forms by $\Upsilon$, so that
$\cv_{F_\alpha}F_\beta=\Upsilon_{\alpha\beta}^\gamma F_\gamma$. Note that these
are not Christoffel symbols with respect to a chart; in particular, it is not true in general
that $\Upsilon_{\alpha\beta}^\gamma=\Upsilon_{\beta\alpha}^\gamma$. However, we do
have $\Upsilon_{0i}^\gamma=\Upsilon_{i0}^\gamma$, because $[F_0,F_i]=0$, a fact
that we use below.

The only forms $\Upsilon$ for $G$ that differ from those $\up\Upsilon$ for $\up G$ are
\begin{equation}
 \label{connection}
 \Upsilon_{0i}^i = \frac12 G^{ii}\Ds(G_{ii})\qquad\text{and}\qquad
 \Upsilon_{ii}^0 = -\frac12\Ds(G_{ii}).
\end{equation}
One obtains these from the calculations
\[
 \Ds G_{ii}=2\lp\cv_{F_0}F_i,F_i\rp=2\Upsilon_{0i}^i G_{ii}
\]
and
\[
 \Upsilon_{ii}^0=\lp\cv_{F_i}F_i,F_0\rp=-\lp F_i,\cv_{F_i}F_0\rp
 =-\lp F_i,\cv_{F_0}F_i\rp=-\Upsilon_{0i}^i G_{ii}.
\]

Hereafter, we only consider vector fields $N$ that satisfy our warped Berger
\emph{Ansatz} $N=N^\alpha(s)F_\alpha$, so that all $F_i(N^0)=0$. It follows
that for any such vector fields $M,N$, one has
\begin{equation}
 \label{covariant}
 \cv_M N -\up\cv_{(\mc V M)}(\mc V N)
 =M^0\Big\{\Ds N^\beta F_\beta+\Upsilon_{0i}^i N^i F_i\Big\}
 +M^i\Big\{\Upsilon_{ii}^0 N^i F_0 + \Upsilon_{i0}^i N^0 F_i\Big\}.
\end{equation}

\subsubsection{Curvatures of vertical planes}
O'Neill's tensor $T$ encodes the second fundamental form of the fibers $\mc F_x$ ---
that is to say, it encodes $\cv-\up\cv$.
We write $T_M N=M^\alpha N^\beta T_{\alpha\beta}^\gamma F_\gamma$.
Formulas~\eqref{define-T}--\eqref{covariant}
show that for vertical vector fields $U,V$, one has
$\cv_U V-\up\cv_U V = U^i V^j T_{ij}^\alpha F_\alpha$, where
all components $T_{ij}^\alpha$ vanish except $T_{ii}^0=\Upsilon_{ii}^0$, which have
the values
\[
 T_{11}^0=-ff_s,\qquad T_{22}^0=-gg_s,\quad\text{and}\quad T_{33}^0=-hh_s.
\]
O'Neill's formula for $\Rm$ applied to vertical vector fields $U,V,W,P$,
\[
 \lp R(U,V)W,P\rp = \lp\up R(U,V)W,P\rp +\lp T_U W, T_V P\rp-\lp T_V W,T_U P\rp,
\]
thus implies that the sectional curvatures of the vertical planes are
\begin{align*}
 \kappa_{12}&=\up\kappa_{12}-\frac{f_s g_s}{fg},\\
 \kappa_{23}&=\up\kappa_{23}-\frac{g_s h_s}{gh},\\
 \kappa_{31}&=\up\kappa_{31}-\frac{f_s h_s}{fh}.\\
\end{align*}

\subsubsection{Curvature of mixed planes}
Because $\mc B$ is one-dimensional, the only other curvatures we need consider are
those involving planes $F_0\wedge F_i$. For the same reason, O'Neill's tensor $A$,
which measures the  obstruction to integrability of the distribution $\mc H$, vanishes.
These observations reduce the remaining curvature formulas to
\begin{align}
 \label{mixed}
 \lp R(F_0,U)V,F_0\rp&=\lp(\cv_{F_0}T)_U V,F_0\rp-\lp T_U F_0, T_V F_0\rp,\\
 \label{verymixed}
 \lp R(U,V)W,F_0\rp&=\lp(\cv_U T)_V W,F_0\rp-\lp(\cv_V T)_U W,F_0\rp,
\end{align}
where $U,V,W$ are again vertical vector fields.

To compute the curvatures given by ~\eqref{mixed}, we first use~\eqref{define-T} and
\eqref{covariant} to see that
\[
 T_U F_0 = \mc V\Big(\cv_U F_0\Big) = U^i \Upsilon_{i0}^i F_i.
\]
Next we observe that $\cv_0 T_{ij}^k$=0 for all $i,j,k$, and that $\cv_0 T_{ij}^0=0$ for $i\neq j$,
while
\[
 \cv_0 T_{ii}^0=\Ds T_{ii}^0-2\Upsilon_{0i}^i T_{ii}^0.
\]
It thus follows from~\eqref{mixed} that $R_{0ij0}=0$ for all $i\neq j$, while the nonvanishing
sectional curvatures $G^{ii} R_{0ii0}$ are
\[
 \kappa_{01}=-\frac{f_{ss}}{f},\qquad
 \kappa_{02}=-\frac{g_{ss}}{g},\quad\text{and}\quad
 \kappa_{03}=-\frac{h_{ss}}{h}.
\]

Rather than use~\eqref{verymixed} to compute the remaining curvatures, it is easier to
proceed as follows. To study this \emph{Ansatz} under Ricci flow, it suffices to compute
 $\Rc$, and the only remaining curvatures one needs to accomplish this are all elements
 of the form $R_{0jjk}$ with $j\neq k$. By definition of $\Rm$, one has
\begin{align*}
 R(F_\alpha,F_\beta)F_\gamma
 &=\cv_{F_\alpha}(\cv_{F_\beta}F_\gamma)-\cv_{F_\beta}(\cv_{F_\alpha}F_\gamma)
 	-\cv_{[F_\alpha,F_\beta] }F_\gamma\\
 & =F_\alpha\big(\Upsilon_{\beta\gamma}^\lambda\big)F_\lambda
 	+\Upsilon_{\beta\gamma}^\mu\Upsilon_{\alpha\mu}^\lambda F_\lambda
         -F_\beta(\Upsilon_{\alpha\gamma}^\lambda)F_\lambda
         -\Upsilon_{\alpha\gamma}^\mu\Upsilon_{\beta\mu}^\lambda F_\lambda
         -\cv_{[F_\alpha,F_\beta]}F_\gamma.
\end{align*}
Therefore, because $[F_0,F_j]=0$ and $j\neq k$, we obtain
\[
 R_{0jj}^k=\Upsilon_{jj}^\mu\Upsilon_{0\mu}^k-\Upsilon_{0j}^\mu\Upsilon_{j\mu}^k
 =\up\Upsilon_{jj}^k\big(\Upsilon_{0k}^k-\Upsilon_{0j}^j\big).
\]
But $\up\Gamma_{jj}^k=\up G^{kk}\lp\up\cv_{F_j}F_j,F_k\rp$, and
$\up\cv_{F_j}F_j=-(\mr{ad}F_j)^* F_j=0$, because $(F_1,F_2,F_3)$ is a Milnor frame
(see \cite[Chapter~1.4]{CK04}). Therefore, all $R_{0jjk}=0$, and so $R_{0k}=0$.

\subsection{Evolution of warped Berger metrics by Ricci flow}
The calculations in Section~\ref{ComputeCurvatures} show that the Ricci
endormorphism is diagonal in the coordinates induced by $(F_0,\dots,F_3)$,
with $R_\alpha^\gamma=0$ if $\gamma\neq\alpha$, and
$R_\alpha^\alpha=\sum_{\beta\neq\alpha}\kappa_{\alpha\beta}$.
Hence the warped Berger metric \emph{Ansatz} is preserved under Ricci flow.

We now abuse notation and allow $f,g,h$ and the gauge $\rho$ (hence $s$) to depend on time as well
as the spatial variable $\xi$. Then Ricci flow of $G$ is equivalent to the system
\begin{subequations}		\label{RFsystem}
\begin{align}
 f_t&=f_{ss}+\left(\frac{g_s}{g}+\frac{h_s}{h}\right)f_s-f\big(\up\kappa_{12}+\up\kappa_{31}\big),\\
 g_t&=g_{ss}+\left(\frac{f_s}{f}+\frac{h_s}{h}\right)g_s-g\big(\up\kappa_{12}+\up\kappa_{23}\big),\\
 h_t&=h_{ss}+\left(\frac{f_s}{f}+\frac{g_s}{g}\right)h_s-h\big(\up\kappa_{23}+\up\kappa_{31}\big),
\end{align}
\end{subequations}
along with the evolution equation
\begin{equation}	\label{DistanceEvolution}
 (\log\rho)_t=-\big(\kappa_{01}+\kappa_{02}+\kappa_{03}\big)
\end{equation}
satisfied by the gauge $\rho$. Our choice of gauge means that space and time derivatives do
not commute; instead one has the commutator
\[
    \left[\frac{\partial}{\partial t},\frac{\partial}{\partial s}\right]
    =\big(\kappa_{01}+\kappa_{02}+\kappa_{03}\big)\frac{\partial}{\partial s}.
\]

Finally, one has to impose boundary conditions at $\xi_\pm$ in order to get a smooth metric
on some topology. In this paper, we study metrics on $\mc S^1\times\mc S^3$, so we take
$[\xi_-,\xi_+]=[-\pi,\pi]$ and stipulate that everything in sight is $2\pi$-periodic in space. This
allows us in the body of the paper to regard $\xi$ as a coordinate on $\mc S^1$, with $s$
representing arclength from a fixed but arbitrary point $\xi_0$.

\subsection{A simplified \emph{Ansatz}}

We conjecture that Ricci flow solutions satisfying the general
system~\eqref{RFsystem} become asymptotically rotationally symmetric
if they develop neckpinch singularities. In this paper, we prove the conjecture in the special
case that $g=h$ initially. The following result shows that this condition is preserved.

\begin{lemma} \label{Berger}
For these metrics, any ordering, e.g., $f\leq g$ or $g\leq h$, that holds initially is preserved by Ricci flow.
\end{lemma}
\begin{proof}
Without loss of generality, it suffices to show that the condition $g\leq h$ is preserved.
We set $z:=g-h$. Then a straightforward computation shows that
\[
 z_t = z_{s}+\frac{f_s}{f}z_s
 +\left\{\frac{g_s h_s}{gh}-\up\kappa_{23}+\zeta\right\}z,
\]
where
\begin{multline*}
 \zeta=2\frac{f^2-gh}{f^2gh}-3\frac{g^2+gh+h^2}{f^2gh}\\
  -\frac{f^4-2f^2(g^2+gh+h^2)+g^4+g^3h+g^2h^2+gh^3+h^4}{(fgh)^2}.
\end{multline*}
By the parabolic maximum principle, the condition $z\leq0$ is preserved if it
holds initially; the same is true of $z\geq0$.
\end{proof}

\begin{remark}
Unsurprisingly,  ordering is also preserved by the Ricci flow \textsc{ode} system on
$\mr{SU}(2)$; see \cite[Chapter~1.5]{CK04}.
\end{remark}

\section{Initial data that result in local singularities}		\label{InitialData}

Here we show that there are (non-unique) open sets of warped Berger initial
data giving rise to solutions that satisfy Assumption~\ref{MildAssumption} or
Assumption~\ref{StrongerAssumption} and that develop local neckpinch singularities.

To begin, we consider metrics of the form~\eqref{SimpleAnsatz} on
$\mb R\times\mc S^3$, with $g(s)=\gamma(s)$ and $f=\eta\gamma(s)$, where
\[
\gamma(s):=\sqrt{\alpha+\beta s^2}.
\]
Here, $\alpha$, $\beta$, and $\eta\leq1$ are positive constants. It is easy to
check that $\gamma_s=\beta s/\gamma$ satisfies the bound $|\gamma_s|\leq\sqrt \beta$,
and that $\gamma_{ss}=\alpha\beta/\gamma^3$. Then, observing that the curvatures of
these metrics are 
\[
\kappa_{12}=\kappa_{31}=\frac{\eta^2\gamma^2-\beta^2s^2}{\gamma^4}\qquad\mbox{and}
\qquad\kappa_{23}=\frac{(4-3\eta^2)\gamma^2-\beta^2s^2}{\gamma^4},
\]
and 
\[
\kappa_{01}=\kappa_{02}=\kappa_{03}=-\frac{\alpha\beta}{\gamma^4},
\]
one computes easily that the scalar curvature is
\[
R=\frac{(4-\eta^2-3\beta)}{\gamma^2},
\]
which is positive if $\beta$ is sufficiently small.

Next, working on $\mc S^1\times\mc S^3$, we set $\rho=\Lambda$, where $\Lambda$
is a large constant. It follows that  $s\in[-\Lambda\pi,\Lambda\pi]$. We now choose 
$g(s)=\bar\gamma(s)$ and $f=\eta\bar\gamma(s)$, where $\bar\gamma$ is the piecewise
smooth function
\[
\bar\gamma(s):=\left\{
\begin{matrix}
\gamma(s) & \mbox{if }|s|\leq\Lambda,\\ \\
\gamma(\Lambda) & \mbox{if }|s|>\Lambda.
\end{matrix}
\right.
\]
For $|s|>\Lambda$, one has $R=(4-\eta^2)/\gamma(\Lambda)^2$, which is positive.
(As noted above, $R$ is also positive  for $|s|\leq\Lambda$.)

We smooth the ``corner'' that $\bar\gamma$ has at $s=\Lambda$ in two steps. First we construct
$\tilde\gamma$, which agrees with $\bar\gamma$ outside intervals
$I_\delta:=\{|s|\in(\Lambda-\delta,\Lambda+\delta)\}$ and has $\tilde\gamma_{ss}$ constant
in each $I_\delta$. We choose the constant
$-\beta(\Lambda-\delta)/\big(2\delta\sqrt{\alpha+\beta(\Lambda-\delta)^2}\big)$,
so that $\tilde\gamma$ is $C^1$. Because
$\tilde\gamma_{ss}<0$ in each $I_\delta$ and $|\tilde\gamma_s|\leq\sqrt\beta$ everywhere, the
metric induced by $\tilde\gamma$ continues to have positive scalar curvature everywhere it is smooth.

Now $\tilde\gamma$ is piecewise smooth, and $\tilde\gamma_{ss}$ has simple jump
discontinuities at $|s|=\Lambda\pm\delta$. So in the final step, we smooth $\tilde\gamma$,
obtaining a $C^\infty$ function $\tilde{\tilde\gamma}$ that agrees with $\bar\gamma$ outside
intervals $I_{2\delta}$. It is clear that this can be done so that
$|\tilde{\tilde\gamma}_s|\leq2\sqrt\beta$. For $\alpha\in(0,\alpha^*)$, $\beta\in(0,\beta^*)$,
and $\eta\in(0,1)$, this produces a family $\mc G$ of initial data with $f<g$, positive scalar
curvature, and uniform curvature bounds, depending only on $\alpha^*$, $\beta^*$, and $\delta$.
Hence the first two conditions of Assumption~\ref{MildAssumption} are satisfied.

Each initial metric $G_0\in\mc G$ has a ``pseudo-neck'' at $s=0$ of radius $\eta\alpha$ and a ``pseudo-bump''
at $|s|=\Lambda\pi$ of height $\gamma(\Lambda)$. A solution originating from $G_0$ must 
become singular at some $T<\infty$ and thus must satisfy the third condition of
Assumption~\ref{MildAssumption}; indeed, it follows from Lemma~\ref{lem-lower} that the singular
time $T$ satisfies
\[
T\leq\frac{\dn M^2(0)}{c}=\frac{\eta^2\alpha}{c}.
\]

Finally, we note that it follows from Lemma~\ref{UniformlyLipschitz} and Corollary~\ref{NotTooFast}
that for $s\geq\Lambda+2\delta$, one has
\[
f^2(s,T)\geq\eta^2\left\{\alpha-\frac{C}{c}\alpha+\beta\Lambda^2\right\}.
\]
We observe that the constants $c$ and $C$ from Lemma~\ref{UniformlyLipschitz}
and Corollary~\ref{NotTooFast}, respectively,  depend only on the ratio $f/g=\eta$ and on bounds
for the curvatures, all of which are independent of $\Lambda\ge1$. So by taking $\Lambda$
sufficiently large, we can ensure that $f^2(s,T)>0$, hence that the singularity is local.

\begin{remark}
It is clear from this construction that there is a neighborhood $\mc G_1$ of $\mc G$ in
$\mathfrak{Met}(\mc S^1\times\mc S^3)$ such that all warped Berger solutions originating in
this open set satisfy Assumption~\ref{MildAssumption} and develop local singularities in finite time.
\end{remark}

\begin{remark}	\label{WeCanDoIt}
It is also clear from the construction that by taking $\beta$ sufficiently close to $0$ and
$\eta$ sufficiently close to $1$, we obtain a family $\mc G'$ of initial data that satisfy
Assumption~\ref{StrongerAssumption}, as do all warped Berger solutions originating
in a neighborhood $\mc G_2$ of $\mc G'$ in $\mathfrak{Met}(\mc S^1\times\mc S^3)$.
Because our construction is reflection-symmetric, it also produces initial data that satisfy
Assumption~\ref{Reflection}.
\end{remark}

\section{Parabolically rescaled equations}	\label{FormalNeckpinch}
\subsection{Evolution equations in blow-up variables}

Given a singularity time $T$, we introduce parabolically dilated time and space variables
\[
\tau:=-\log(T-t)\qquad\mbox{and}\qquad\sigma:=e^{\tau/2}s,
\]
respectively. We parabolically dilate the metric, considering\,\footnote{The
fraction $\frac12$ corresponds to the factor $1/\sqrt{2(n-1)}$ in \cite{AK07} and
simplifies what follows.}
\[
u:=\frac12 e^{\tau/2}f\qquad\mbox{and}\qquad v:=\frac12e^{\tau/2}g.
\]

One computes that
$f_t=2e^{\tau/2}\big\{u_\tau+\sigma_\tau u_\sigma-\frac12 u\big\}$,
$f_s=2u_\sigma$, and $f_{ss}=2e^{\tau/2}u_{\sigma\sigma}$, where
$\sigma_\tau=\frac12\sigma+\mc I$, with $\mc I$ the nonlocal term
\[
\mc I(\sigma,\tau) = \int_0^\sigma \left(\frac{u_{\bar\sigma \bar\sigma}}{u}
	+2\frac{v_{\bar\sigma \bar\sigma}}{v}\right)\mr d\bar\sigma.
\]
The nonlocal quantity $\mc I$ is necessary for $\sigma$ and $\tau$ to be
commuting variables, i.e.~for us to interpret $\tau$ derivatives as time derivatives
taken with $\sigma$ rather than $\xi$ fixed. Of course, analogous formulas
hold for $g_t$, $g_s$, and $g_{ss}$.

With  these rescalings imposed , system~\eqref{SRFsystem} becomes
\begin{align*}
u_\tau&=u_{\sigma\sigma}-\left(\frac{\sigma}{2}+\mc I\right)u_\sigma
	+2\frac{v_\sigma}{v}u_\sigma+\frac12\left(u-\frac{u^3}{v^4}\right),\\
v_\tau&=v_{\sigma\sigma}-\left(\frac{\sigma}{2}+\mc I\right)v_\sigma
	+\left(\frac{u_\sigma}{u}+\frac{v_\sigma}{v}\right)v_\sigma
	+\frac12\left(v-\frac{2v^2-u^2}{v^3}\right),
\end{align*}
which reduces to the equation studied in \cite{AK07} if $u=v$.

\subsection{Linearization at the cylinder}
In a space-time neighborhood of the singular set, our results in Lemma~\ref{Cylinder}
show that the solution is close to the self-similarly shrinking cylinder soliton near the
developing neckpinch, so that $u\approx1$ and $v\approx1$. Accordingly, we introduce
(locally small) quantities $\vp$ and $\psi$ defined by
\[
\vp:=u-1 \qquad\text{and}\qquad \psi:=v-1.
\]

Linearizing near $u=1$ and $v=1$, one finds that
\[
\vp_\tau=\Big\{\vp_{\sigma\sigma}-\frac{\sigma}{2}\vp_\sigma-\vp+2\psi\Big\}+N_1(\vp,\psi),
\]
where the operator in braces is linear, and  $N_1(\vp,\psi)$ is the nonlinear term
\begin{multline*}
N_1=\left(\frac{2\psi_\sigma}{1+\psi}-\mc I\right)\vp_\sigma
+(\vp-2\psi)\frac{(1+\psi)^4-1}{(1+\psi)^4}\\
+\frac{-3\vp^2+4\vp\psi+6\psi^2-\vp^3+6\vp\psi^2+4\psi^3+4\vp\psi^3+\psi^4+\vp\psi^4}
{2(1+\psi)^4}.
\end{multline*}
In the same way, one finds that
\[
\psi_\tau=\Big\{\psi_{\sigma\sigma}-\frac{\sigma}{2}\psi_\sigma+\vp\Big\}+N_2(\vp,\psi),
\]
where the operator in braces is linear, and
\[
N_2=\left(\frac{\vp_\sigma}{1+\vp}+\frac{\psi_\sigma}{1+\psi}-\mc I\right)\psi_\sigma
-\vp\frac{(1+\psi)^3-1}{(1+\psi)^3}+\frac{\vp^2+4\psi^2+4\psi^3+\psi^4}
{2(1+\psi)^3}.
\]
For clarity of exposition, we have not simplified $N_1,N_2$ as much as possible here.

Examination of the linearized system for $\vp$ and $\psi$ reveals that it is easily
decoupled by introducing new quantities
\begin{equation}	\label{Define-xy}
x:=\vp-\psi\qquad\text{and}\qquad y=\vp+2\psi.
\end{equation}
Then, neglecting nonlinear terms, one finds that $x$ and $y$
evolve by
\begin{align*}
x_\tau&=x_{\sigma\sigma}-\frac{\sigma}{2}x_\sigma-2x+\cdots=(\mc A-3)x+\cdots,\\
y_\tau&=y_{\sigma\sigma}-\frac{\sigma}{2}y_\sigma+y+\cdots=\quad\mc Ay+\cdots,
\end{align*}
where $\mc A$ is the elliptic operator the generates the quantum harmonic oscillator,
\[
\mc A=\frac{\partial^2}{\partial\sigma^2}
	-\frac{\sigma}{2}\frac{\partial}{\partial\sigma}+1.
\]
The spectrum of $-\mc A$ is $\{\mu_k=\frac{k}{2}-1,\;k\geq0\}$, with associated
eigenfunctions the Hermite polynomials $h_k$, normalized here so that
$h_k(\sigma)=\sigma^k+\mc O(\sigma^{k-2})$. Clearly, the spectrum of $-\mc A+3$ is
$\{\nu_k=\frac{k}{2}+2,\;k\geq0\}$.

\end{document}